\pdfoutput=1
\documentclass[12pt,a4paper,reqno]{amsart}
\usepackage{amssymb}
\usepackage{amscd}
\usepackage{graphicx}
\numberwithin{equation}{section}

\theoremstyle{plain}

\newtheorem{theorem}{Theorem}[section]
\newtheorem{proposition}[theorem]{Proposition}
\newtheorem{lemma}[theorem]{Lemma}
\newtheorem{corollary}[theorem]{Corollary}
\newtheorem{conjecture}[theorem]{Conjecture}

\theoremstyle{definition}

\newtheorem{definition}[theorem]{Definition}
\newtheorem{remark}[theorem]{Remark}

\newtheorem{example}[theorem]{Example}

\newcommand\R{\mathbb{R}}
\newcommand\Z{\mathbb{Z}}
\newcommand\N{\mathbb{N}}

\begin{document}

\title{Density Hales-Jewett and Moser numbers}

\author{D.H.J. Polymath}
\address{http://michaelnielsen.org/polymath1/index.php}


\begin{abstract}  
For any $n \geq 0$ and $k \geq 1$, the \emph{density Hales-Jewett number} $c_{n,k}$ is defined as the size of the largest subset of the cube $[k]^n$ := $\{1,\ldots,k\}^n$ which contains no combinatorial line; similarly, the Moser number $c'_{n,k}$ is the largest subset of the cube $[k]^n$ which contains no geometric line.  A deep theorem of Furstenberg and Katznelson \cite{fk1}, \cite{fk2}, \cite{mcc} shows that $c_{n,k}$ = $o(k^n)$ as $n \to \infty$ (which implies a similar claim for $c'_{n,k}$); this is already non-trivial for $k = 3$. Several new proofs of this result have also been recently established \cite{poly}, \cite{austin}.

Using both human and computer-assisted arguments, we compute several values of $c_{n,k}$ and $c'_{n,k}$ for small $n,k$. For instance the sequence $c_{n,3}$ for $n=0,\ldots,6$ is $1,2,6,18,52,150,450$, while the sequence $c'_{n,3}$ for $n=0,\ldots,6$ is $1,2,6,16,43,124,353$. We also prove some results for higher $k$, showing for instance that an analogue of the LYM inequality (which relates to the $k = 2$ case) does not hold for higher $k$, and also establishing the asymptotic lower bound $c_{n,k} \geq k^n \exp\left( - O(\sqrt[\ell]{\log n})\right)$ where $\ell$ is the largest integer such that $2k > 2^\ell$. 
\end{abstract}

\maketitle


\section{Introduction}

For any integers $k \geq 1$ and $n \geq 0$, let $[k] := \{1,\ldots,k\}$, and define $[k]^n$ to be the cube of words of length $n$ with alphabet in $[k]$.  Thus for instance $[3]^2 = \{11,12,13,21,22,23,31,32,33\}$.

We define a \emph{combinatorial line} in $[k]^n$ to be a set of the form $\{ w(i): i = 1,\ldots,k\} \subset [k]^n$, where $w \in ([k] \cup \{x\})^n \backslash [k]^n$ is a word of length $n$ with alphabet in $[k]$ together with a ``wildcard'' letter $x$ which appears at least once, and $w(i) \in [k]^n$ is the word obtained from $w$ by replacing $x$ by $i$; we often abuse notation and identify $w$ with the combinatorial line $\{ w(i): i = 1,\ldots,k\}$ it generates.  Thus for instance, in $[3]^2$ we have $x2 = \{12,22,32\}$ and $xx = \{11,22,33\}$ as typical examples of combinatorial lines. In general, $[k]^n$ has $k^n$ words and $(k+1)^n-k^n$ lines.

\begin{figure}[tb]
\centerline{\includegraphics[height=6cm,width=10cm]{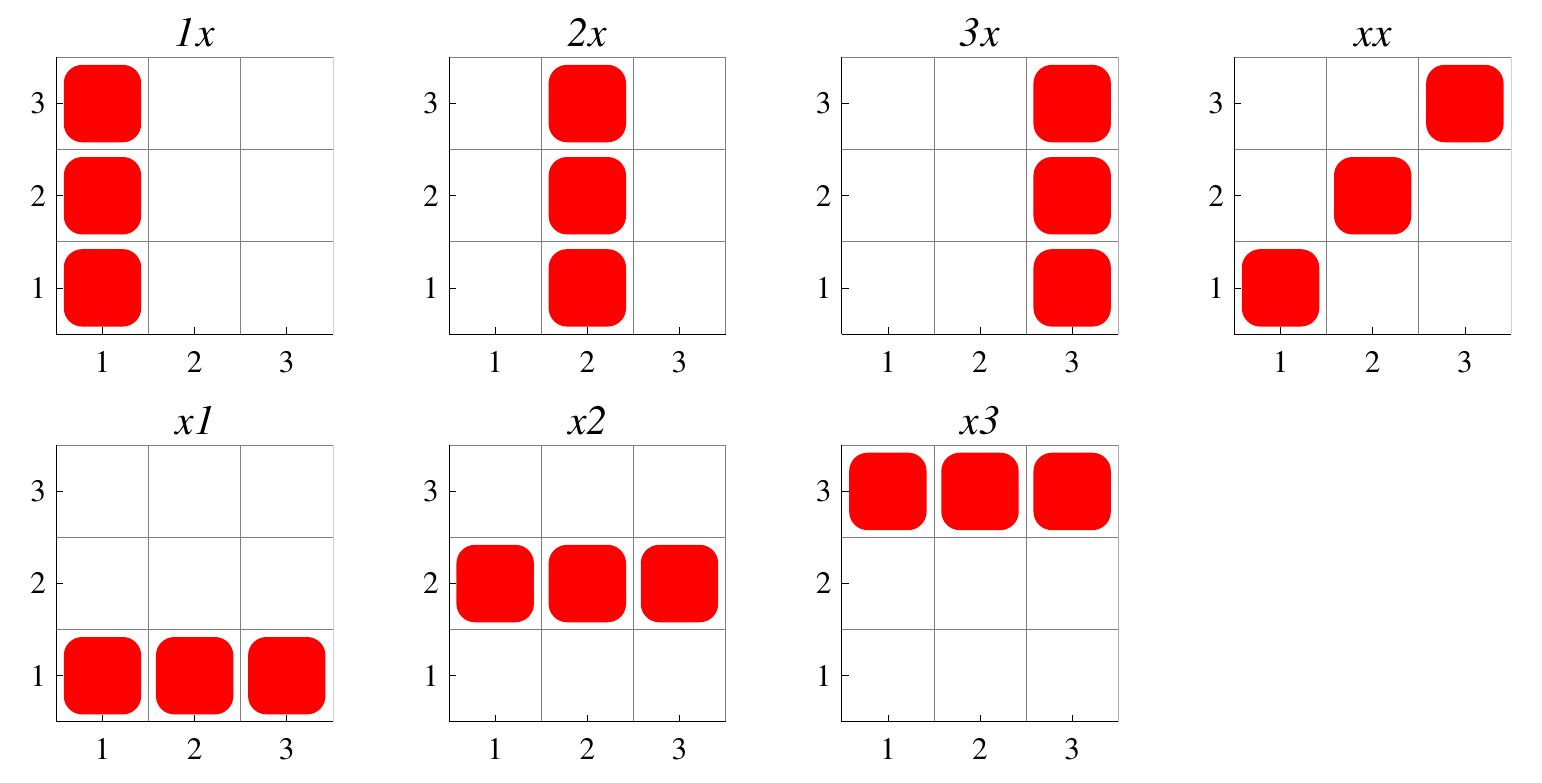}}
\caption{Combinatorial lines in $[3]^2$.}
\label{fig-line}
\end{figure}

A set $A \subset [k]^n$ is said to be \emph{line-free} if it contains no combinatorial lines.  Define the \emph{$(n,k)$ density Hales-Jewett number} $c_{n,k}$ to be the maximum cardinality $|A|$ of a line-free subset of $[k]^n$.  Clearly, one has the trivial bound $c_{n,k} \leq k^n$.  A deep theorem of Furstenberg and Katznelson~\cite{fk1}, \cite{fk2} asserts that this bound can be asymptotically improved:

\begin{theorem}[Density Hales-Jewett theorem]\label{dhj}  For any fixed $k \geq 2$, one has $\lim_{n \to \infty} c_{n,k}/k^n = 0$.
\end{theorem}

\begin{remark} The difficulty of this theorem increases with $k$.  For $k=1$, one clearly has $c_{n,1}=1$.  For $k=2$, a classical theorem of Sperner~\cite{sperner} asserts, in our language, that $c_{n,2} = \binom{n}{\lfloor n/2\rfloor}$.  The case $k=3$ is already non-trivial (for instance, it implies Roth's theorem~\cite{roth} on arithmetic progressions of length three) and was first established in \cite{fk1} (see also \cite{mcc}).  The case of general $k$ was first established in~\cite{fk2} and has a number of implications, in particular implying Szemer\'edi's theorem~\cite{szem} on arithmetic progressions of arbitrary length.
\end{remark}

The Furstenberg-Katznelson proof of Theorem~\ref{dhj} relies on ergodic-theory techniques and does not give an explicit decay rate for $c_{n,k}$.  Recently, two further proofs of this theorem have appeared, by Austin~\cite{austin} and by the sister Polymath project to this one~\cite{poly}.  The proof of~\cite{austin} also uses ergodic theory, but the proof in~\cite{poly} is combinatorial and gave effective bounds for $c_{n,k}$ in the limit $n \to \infty$. For example, if $n$ can be written as an exponential tower $2 \uparrow 2 \uparrow 2 \uparrow \ldots \uparrow 2$ with $m$ 2s, then $c_{n,3} \ll 3^n m^{-1/2}$.  However, these bounds are not believed to be sharp, and in any case are only non-trivial in the asymptotic regime when $n$ is sufficiently large depending on $k$.

Our first result is the following asymptotic lower bound. The construction is based on the recent refinements \cite{elkin,greenwolf,obryant} of a well-known construction of Behrend~\cite{behrend} and Rankin~\cite{rankin}. The proof of Theorem~\ref{dhj-lower} is in Section~\ref{dhj-lower-sec}. Let $r_k(n)$ be the maximum size of a subset of $[n]$ that does not contain a $k$-term arithmetic progression.

\begin{theorem}[Asymptotic lower bound for $c_{n,k}$]\label{dhj-lower}  For each $k\geq 3$, there is an absolute constant $C>0$ such that 
  \[ c_{n,k} \geq C k^n \left(\frac{r_k(\sqrt{n})}{\sqrt{n}}\right)^{k-1} = k^n \exp\left( - O(\sqrt[\ell]{\log n}) \right), \]
where $\ell$ is the largest integer satisfying $2k>2^{\ell}$. More specifically,
  \[ c_{n,k} \geq C k^{n-\alpha(k) \sqrt[\ell]{\log n} + \beta(k) \log\log n},\]
where all logarithms are base-$k$, and $\alpha(k) = (\log 2)^{1-1/\ell} \ell 2^{(\ell-1)/2-1/\ell}$ and $\beta(k)=(k-1)/(2\ell)$.
\end{theorem}

In the case of small $n$, we focus primarily on the first non-trivial case $k=3$.  We have computed the following explicit values of $c_{n,3}$ (entered in the OEIS~\cite{oeis} as A156762):

\begin{theorem}[Explicit values of $c_{n,3}$]\label{dhj-upper}  We have $c_{0,3} = 1$, $c_{1,3} = 2$, $c_{2,3} = 6$, $c_{3,3} = 18$, $c_{4,3} = 52$, $c_{5,3}=150$, and $c_{6,3}=450$.  
\end{theorem}

This result is established in Sections~\ref{dhj-lower-sec},~\ref{dhj-upper-sec}.  Initially these results were established by an integer program, but we provide completely computer-free proofs here.  The constructions used in Section~\ref{dhj-lower-sec} give reasonably efficient constructions for larger values of $n$; for instance, they show that $3^{99} \leq c_{100,3} \leq 2 \times 3^{99}$.  See Section~\ref{dhj-lower-sec} for further discussion.


A variant of the density Hales-Jewett theorem has also been studied in the literature.  Define a \emph{geometric line} in $[k]^n$ to be any set of the form $\{ a+ir: i=1,\ldots,k\}$ in $[k]^n$, where we identify $[k]^n$ with a subset of $\Z^n$, and $a, r \in \Z^n$ with $r \neq 0$.  Equivalently, a geometric line takes the form $\{ w( i, k+1-i ): i =1,\ldots,k \}$, where $w \in ([k] \cup \{x,\overline{x}\})^n \backslash [k]^n$ is a word of length $n$ using the numbers in $[k]$ and two wildcards $x, \overline{x}$ as the alphabet, with at least one wildcard occurring in $w$, and $w(i,j) \in [k]^n$ is the word formed by substituting $i,j$ for $x,\overline{x}$ respectively.  Figure~\ref{fig-geomline} shows the eight geometric lines in $[3]^2$.  Clearly every combinatorial line is a geometric line, but not conversely.  In general, $[k]^n$ has $((k+2)^n-k^n)/2$ geometric lines. 

\begin{figure}[tb]
\centerline{\includegraphics[height=6cm,width=10cm]{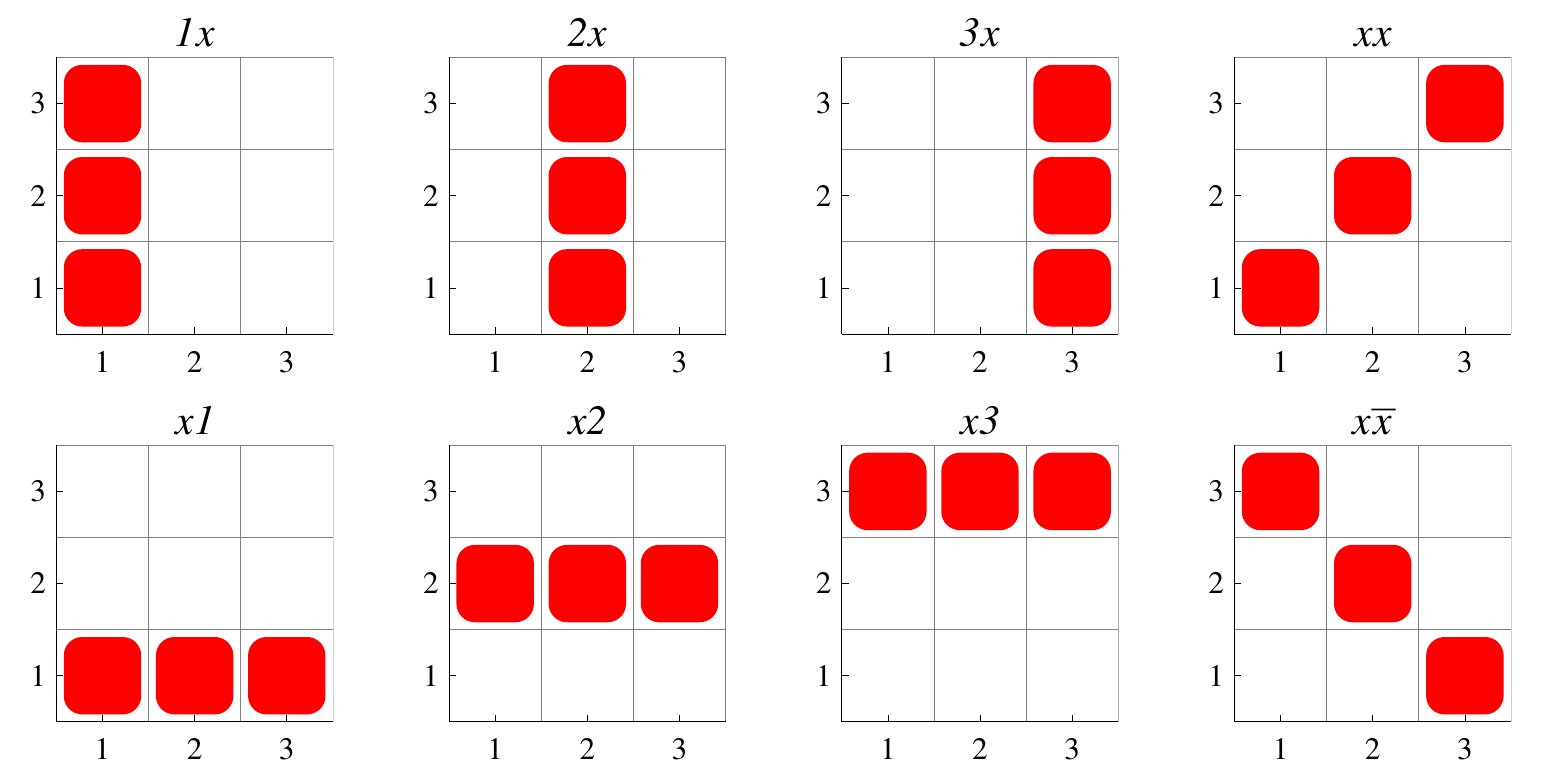}}
\caption{Geometric lines in $[3]^2$.}
\label{fig-geomline}
\end{figure}

Define a \emph{Moser set} in $[k]^n$ to be a subset of $[k]^n$ that contains no geometric lines, and let $c'_{n,k}$ be the maximum cardinality $|A|$ of a Moser set in $[k]^n$.  Clearly one has $c'_{n,k} \leq c_{n,k}$, so in particular from Theorem~\ref{dhj} one has $c'_{n,k}/k^n \to 0$ as $n \to \infty$.  (Interestingly, there is no known proof of this fact that does not go through Theorem~\ref{dhj}, even for $k=3$.)  Again, $k=3$ is the first non-trivial case: it is clear that $c'_{n,1}=0$ and $c'_{n,2}=1$ for all $n$.

The question of computing $c'_{n,3}$ was first posed by Moser~\cite{moser}.  Prior to our work, the values
$$ c'_{0,3}=1; c'_{1,3}=2; c'_{2,3}=6; c'_{3,3}=16; c'_{4,3}=43$$
were known~\cite{chvatal2},~\cite{chandra} (this is Sequence A003142 in the OEIS~\cite{oeis}).  We extend this sequence slightly:

\begin{theorem}[Values of $c'_{n,3}$ for small $n$]\label{moser}  We have $c'_{0,3} = 1$, $c'_{1,3} = 2$, $c'_{2,3} = 6$, $c'_{3,3} = 16$, $c'_{4,3} = 43$, $c'_{5,3} = 124$, and $c'_{6,3} = 353$.
\end{theorem}

This result is established in Sections \ref{moser-lower-sec}, \ref{moser-upper-sec}.  The arguments given here are computer-assisted; however, we have found alternate (but lengthier) computer-free proofs for the above claims with the the exception of the proof of $c'_{6,3}=353$, which requires one non-trivial computation (Lemma \ref{paretop-4}).  These alternate proofs are not given in this paper to save space, but can be found at \cite{polywiki}.

We establish a lower bound for this problem of $(2+o(1))\binom{n}{i}2^i\leq c'_{n,3}$, which is maximized for $i$ near $2n/3$.  This bound is around one-third better than the previous literature\cite{moser}, \cite{chvatal1}.  We also give methods to improve on this construction.

Earlier lower bounds were known.  Indeed, let $A(n,d)$ denote the size of the largest binary code of length $n$
 and minimal distance $d$. Then
\begin{equation}\label{cnchvatalintro}
c'_{n,3}\geq \max_k \left( \sum_{j=0}^k \binom{n}{j} A(n-j, k-j+1)\right).
\end{equation}
which, with $A(n,1)=2^n$ and $A(n,2)=2^{n-1}$,  implies in particular that
\begin{equation}\label{binom}
c'_{n,3} \geq \binom{n+1}{\lfloor \frac{2n+1}{3} \rfloor} 2^{\lfloor \frac{2n+1}{3} \rfloor - 1} 
\end{equation}
for $n \geq 2$.  This bound is not quite optimal; for instance, it gives a lower bound of $c'_{6,3} \geq 344$.  

\begin{remark} Let $c''_{n,3}$ be the size of the largest subset of ${\mathbb F}_3^n$ which contains no lines $x, x+r, x+2r$ with $x,r \in {\mathbb F}_3^n$ and $r \neq 0$, where ${\mathbb F}_3$ is the field of three elements.  Clearly one has $c''_{n,3} \leq c'_{n,3} \leq c_{n,3}$.  It is known that
$$ c''_{0,3}=1; c''_{1,3}=2; c''_{2,3}=4; c'_{3,3}=9; c'_{4,3}=20; c''_{5,3}=45; c''_{6,3} = 112;$$
see \cite{potenchin}.
\end{remark}

As mentioned earlier, the sharp bound on $c_{n,2}$ comes from Sperner's theorem.  It is known that Sperner's theorem can be refined to the \emph{Lubell-Yamamoto-Meshalkin (LYM) inequality}, which in our language asserts that
$$
\sum_{a_1,a_2 \geq 0; a_1+a_2 = n} \frac{|A \cap \Gamma_{a_1,a_2}|}{|\Gamma_{a_1,a_2}|} \leq 1
$$
for any line-free subset $A \subset [2]^n$, where the \emph{cell} $\Gamma_{a_1,\ldots,a_k} \subset [k]^n$ is the set of words in $[k]^n$ which contain exactly $a_i$ $i$'s for each $i=1,\ldots,k$.  It is natural to ask whether this inequality can be extended to higher $k$.  Let $\Delta_{n,k}$ denote the set of all tuples $(a_1,\ldots,a_k)$ of non-negative integers summing to $n$, define a \emph{simplex} to be a set of $k$ points in $\Delta_{n,k}$ of the form
$(a_1+r,a_2,\ldots,a_k), (a_1,a_2+r,\ldots,a_k),\ldots,(a_1,a_2,\ldots,a_k+r)$ for some $0 < r \leq n$ and $a_1,\ldots,a_k$ summing to $n-r$, and define a \emph{Fujimura set}\footnote{Fujimura actually proposed the related problem of finding the largest subset of $\Delta_{n,k}$ that contained no equilateral triangles; see \cite{fuji}.  Our results for Fujimura sets can be found at the page {\tt Fujimura's problem} at \cite{polywiki}.} to be a subset $B \subset \Delta_{n,k}$ which contains no simplices.  Observe that if $w$ is a combinatorial line in $[k]^n$, then
$$ w(1) \in \Gamma_{a_1+r,a_2,\ldots,a_k}, w(2) \in \Gamma_{a_1,a_2+r,\ldots,a_k}, \ldots, w(k) \in \Gamma_{a_1,a_2,\ldots,a_k+r}$$
for some simplex $(a_1+r,a_2,\ldots,a_k), (a_1,a_2+r,\ldots,a_k),\ldots,(a_1,a_2,\ldots,a_k+r)$.  Thus, if $B$ is a Fujimura set, then $A := \bigcup_{\vec a \in B} \Gamma_{\vec a}$ is line-free.  Note also that
$$
\sum_{\vec a \in \Delta_{n,k}} \frac{|A \cap \Gamma_{\vec a}|}{|\Gamma_{\vec a}|} = |B|.
$$
This motivates a ``hyper-optimistic'' conjecture:

\begin{figure}[tb]
\centerline{\includegraphics[height=6cm,width=8cm]{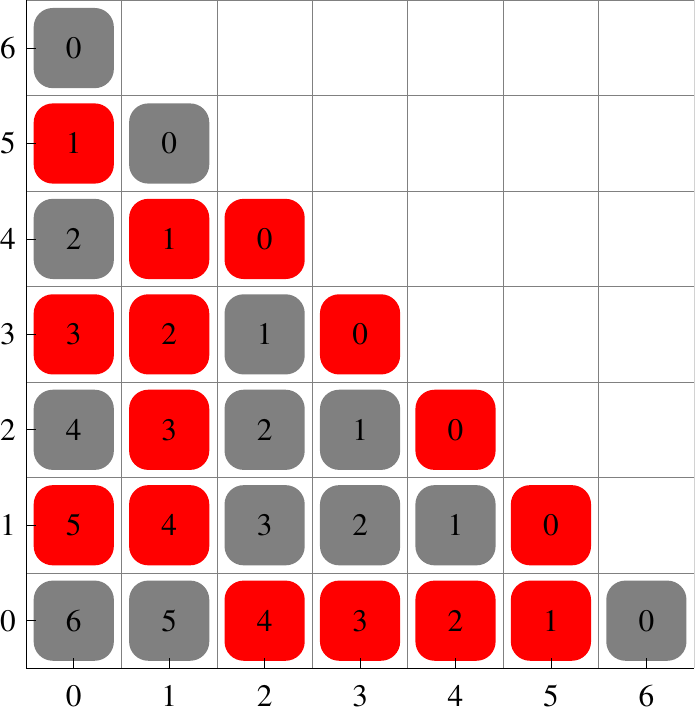}}
\caption{A Fujimura set in $\Delta_{6,3}$, displayed in ``rectangular'' coordinates. The point $(a,b,c)$ is represented by a square at $(a,b)$ labeled with $c$. The Fujimura set is shown in red; its complement in $\Delta_{6,3}$ is shown in gray.}
\label{fig-pic}
\end{figure}

\begin{figure}[tb]
\centerline{\includegraphics[height=6cm,width=6cm]{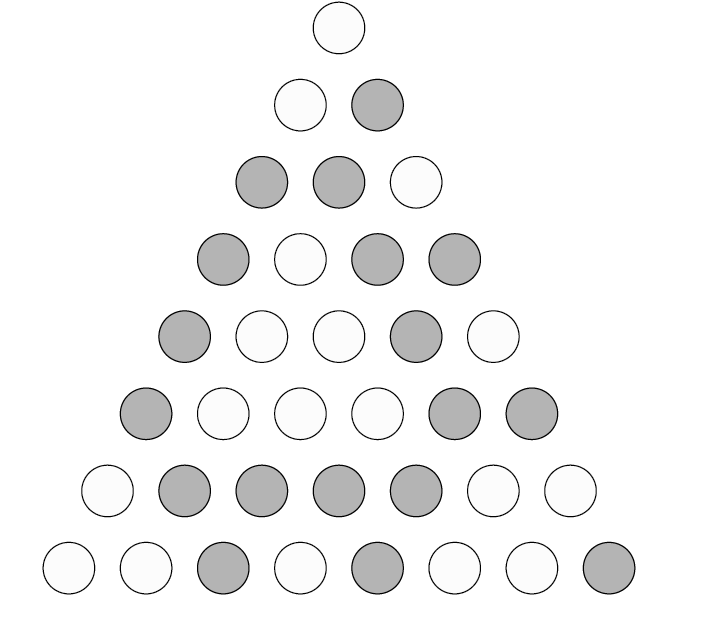}}
\caption{A Fujimura set in $\Delta_{7,3}$, expressed in ``triangular'' coordinates.}
\label{fig-pic2}
\end{figure}

\begin{conjecture}\label{hoc}  For any $k \geq 1$ and $n \geq 0$, and any line-free subset $A$ of $[k]^n$, one has
$$
\sum_{\vec a \in \Delta_{n,k}} \frac{|A \cap \Gamma_{\vec a}|}{|\Gamma_{\vec a}|} \leq c^\mu_{n,k},$$
where $c^\mu_{n,k}$ is the maximal size of a Fujimura set in $\Delta_{n,k}$.
\end{conjecture}

One can show that this conjecture for a fixed value of $k$ would imply Theorem \ref{dhj} for the same value of $k$, in much the same way that the LYM inequality is known to imply Sperner's theorem.  The LYM inequality asserts that Conjecture \ref{hoc} is true for $k \leq 2$.  As far as we know, this conjecture could hold in $k=3$.  However, we found a simple counterexample for $k=4$ and $n=2$, given by the line-free set
$$A := \{(1, 1), (1, 2), (1, 3), (2, 1), (2, 3), (2, 4), (3, 2), (3, 3), (3,4), (4, 1), (4, 2), (4, 4)\}$$
together with the computation that $c^\mu_{4,2}=7$.  It is in fact likely that this conjecture fails for all higher $k$ also.

\subsection{Notation}\label{notation-sec}

There are several subsets of $[k]^n$ which will be useful in our analysis.  We have already introduced combinatorial lines, geometric lines, and cells.  One can generalise the notion of a combinatorial line to that of a \emph{combinatorial subspace} in $[k]^n$ of dimension $d$, which is indexed by a word $w$ in $([k] \cup \{x_1,\ldots,x_d\})^n$ containing at least one of each wildcard $x_1,\ldots,x_d$, and which forms the set $\{ w(i_1,\ldots,i_d): i_1,\ldots,i_d \in [k]\}$, where $w(i_1,\ldots,i_d) \in [k]^d$ is the word formed by replacing $x_1,\ldots,x_d$ with $i_1,\ldots,i_d$ respectively.  Thus for instance, in $[3]^3$, we have the two-dimensional combinatorial subspace $xxy = \{111,112,113,221,222,223,331,332,333\}$.  We similarly have the notion of a \emph{geometric subspace} in $[k]^n$ of dimension $d$, which is defined similarly but with $d$ wildcards $x_1,\ldots,x_d,\overline{x_1},\ldots,\overline{x_d}$, with at least one of either $x_i$ or $\overline{x_i}$ appearing in the word $w$ for each $1 \leq i \leq d$, and the space taking the form $\{ w(i_1,\ldots,i_d,k+1-i_1,\ldots,k+1-i_d): i_1,\ldots,i_d \in [k] \}$.  Thus for instance $[3]^3$ contains the two-dimensional geometric subspace $x\overline{x}y = \{ 131, 132, 133, 221, 222, 223, 311, 312, 313\}$.

An important class of combinatorial subspaces in $[k]^n$ will be the \emph{slices} consisting of $n-1$ distinct wildcards and one fixed coordinate.  We will denote the distinct wildcards here by asterisks, thus for instance in $[3]^3$ we have $2** = \{ 211, 212, 213, 221, 222, 223, 231, 232, 233\}$.  Two slices are \emph{parallel} if their fixed coordinate are in the same position, thus for instance $1**$ and $2**$ are parallel, and one can subdivide $[k]^n$ into $k$ parallel slices, each of which is isomorphic to $[k]^{n-1}$.  In the analysis of Moser slices with $k=3$, we will make a distinction between \emph{centre slices}, whose fixed coordinate is equal to $2$, and \emph{side slices}, in which the fixed coordinate is either $1$ or $3$, thus $[3]^n$ can be partitioned into one centre slice and two side slices.

Another important set in the study of $k=3$ Moser sets are the \emph{spheres} $S_{i,n} \subset [3]^n$, defined as those words in $[3]^n$ with exactly $n-i$ $2$'s (and hence $i$ letters that are $1$ or $3$).  Thus for instance $S_{1,3} = \{ 122, 322, 212, 232, 221, 223\}$.  Observe that $[3]^n = \bigcup_{i=0}^n S_{i,n}$, and each $S_{i,n}$ has cardinality $|S_{i,n}| = \binom{n}{i} 2^{i}$.

It is also convenient to subdivide each sphere $S_{i,n}$ into two components $S_{i,n} = S_{i,n}^o \cup S_{i,n}^e$, where $S_{i,n}^o$ are the words in $S_{i,n}$ with an odd number of $1$'s, and $S_{i,n}^e$ are the words with an even number of $1$'s. Thus for instance $S_{1,3}^o = \{122,212,221\}$ and $S_{1,3}^e = \{322,232,223\}$.  Observe that for $i>0$, $S_{i,n}^o$ and $S_{i,n}^e$ both have cardinality $\binom{n}{i} 2^{i-1}$.

The \emph{Hamming distance} between two words $w,w'$ is the number of coordinates in which $w, w'$ differ, e.g. the Hamming distance between $123$ and $321$ is two.  Note that $S_{i,n}$ is nothing more than the set of words whose Hamming distance from $2\ldots2$ is $i$, which justifies the terminology ``sphere''.

In the density Hales-Jewett problem, there are two types of symmetries on $[k]^n$ which map combinatorial lines to combinatorial lines (and hence line-free sets to line-free sets).  The first is a permutation of the alphabet $[k]$; the second is a permutation of the $n$ coordinates.  Together, this gives a symmetry group of order $k!n!$ on the cube $[k]^n$, which we refer to as the \emph{combinatorial symmetry group} of the cube $[k]^n$.  Two sets which are related by an element of this symmetry group will be called (combinatorially) \emph{equivalent}, thus for instance any two slices are combinatorially equivalent.

For the analysis of Moser sets in $[k]^n$, the symmetries are a bit different.  One can still permute the $n$ coordinates, but one is no longer free to permute the alphabet $[k]$.  Instead, one can \emph{reflect} an individual coordinate, for instance sending each word $x_1 \ldots x_n$ to its reflection $x_1 \ldots x_{i-1} (k+1-x_i) x_{i+1} \ldots x_n$.  Together, this gives a symmetry group of order $2^n n!$ on the cube $[k]^n$, which we refer to as the \emph{geometric symmetry group} of the cube $[k]^n$; this group maps geometric lines to geometric lines, and thus maps Moser sets to Moser sets.  Two Moser sets which are related by an element of this symmetry group will be called (geometrically) \emph{equivalent}.  For instance, a sphere $S_{i,n}$ is equivalent only to itself, and $S_{i,n}^o$, $S_{i,n}^e$ are equivalent only to each other.

\subsection{About this project}

This paper is part of the \emph{Polymath project}, which was launched by Timothy Gowers in February 2009 as an experiment to see if research mathematics could be conducted by a massive online collaboration.  The first project in this series, \emph{Polymath1}, was focused on understanding the density Hales-Jewett numbers $c_{n,k}$, and was split up into two sub-projects, namely an (ultimately successful) attack on the density Hales-Jewett theorem $c_{n,k} = o(k^n)$ (resulting in the paper \cite{poly}), and a collaborative project on computing $c_{n,k}$ and related quantities (such as $c'_{n,k}$) for various small values of $n$ and $k$.  This project (which was administered by Terence Tao) resulted in this current paper.  

Being such a collaborative project, many independent aspects of the problem were studied, with varying degrees of success.  For reasons of space (and also due to the partial nature of some of the results), this paper does not encompass the entire collection of observations and achievements made during the research phase of the project (which lasted for approximately three months).  In particular, alternate proofs of some of the results here have been omitted, as well as some auxiliary results on related numbers, such as coloring Hales-Jewett numbers.  However, these results can be accessed from the web site of this project at \cite{polywiki}.  We are indebted to Michael Nielsen for hosting this web site, which performed a crucial role in the project. A list of contributors to the project (and the grants that supported these individuals) can also be found at this site.

\section{Lower bounds for the density Hales-Jewett problem}\label{dhj-lower-sec}

The purpose of this section is to establish various lower bounds for $c_{n,3}$, in particular establishing Theorem \ref{dhj-lower} and the lower bound component of Theorem \ref{dhj-upper}.

As observed in the introduction, if $B \subset \Delta_{n,3}$ is a Fujimura set (i.e. a subset of $\Delta_{n,3} = \{ (a,b,c) \in \N^3: a+b+c=n\}$ which contains no upward equilateral triangles $(a+r,b,c), (a,b+r,c), (a,b,c+r)$), then the set $A_B := \bigcup_{\vec a \in B} \Gamma_{a,b,c}$ is a line-free subset of $[3]^n$, which gives the lower bound
\begin{equation}\label{cn3}
 c_{n,3} \geq |A_B| = \sum_{(a,b,c) \in B} \frac{n!}{a! b! c!}.
\end{equation}
All of the lower bounds for $c_{n,3}$ in this paper will be constructed via this device.  (Indeed, one may conjecture that for every $n$ there exists a Fujimura set $B$ for which \eqref{cn3} is attained with equality; we know of no counterexamples to this conjecture.)

In order to use \eqref{cn3}, one of course needs to build Fujimura sets $B$ which are ``large'' in the sense that the right-hand side of \eqref{cn3} is large.  A fruitful starting point for this goal is the sets 
$$B_{j,n} := \{ (a,b,c) \in \Delta_{n,3}: a + 2b \neq j \hbox{ mod } 3 \}$$
for $j=0,1,2$.  Observe that in order for a triangle $(a+r,b,c), (a,b+r,c), (a,b,c+r)$ to lie in $B_{j,n}$, the length $r$ of the triangle must be a multiple of $3$.  This already makes $B_{j,n}$ a Fujimura set for $n < 3$	(and $B_{0,n}$ a Fujimura set for $n = 3$).

When $n$ is not a multiple of $3$, the $B_{j,n}$ are all rotations of each other and give equivalent sets (of size $2 \times 3^{n-1}$).  When $n$ is a multiple of $3$, the sets $B_{1,n}$ and $B_{2,n}$ are reflections of each other, but $B_{0,n}$ is not equivalent to the other two sets (in particular, it omits all three corners of $\Delta_{n,3}$); the associated set $A_{B_{0,n}}$ is slightly larger than $A_{B_{1,n}}$ and $A_{B_{2,n}}$ and thus is slightly better for constructing line-free sets.

As mentioned already, $B_{0,n}$ is a Fujimura set for $n \leq 3$, and hence $A_{B_{0,n}}$ is line-free for $n \leq 3$.  Applying \eqref{cn3} one obtains the lower bounds
$$ c_{0,3} \geq 1; c_{1,3} \geq 2; c_{2,3} \geq 6; c_{3,3} \geq 18.$$

For $n>3$, $B_{0,n}$ contains some triangles $(a+r,b,c), (a,b+r,c), (a,b,c+r)$ and so is not a Fujimura set, but one can remove points from this set to recover the Fujimura property.  For instance, for $n \leq 6$, the only triangles in $B_{0,n}$ have side length $r=3$.  One can ``delete'' these triangles by removing one vertex from each; in order to optimise the bound \eqref{cn3} it is preferable to delete vertices near the corners of $\Delta_{n,3}$ rather than near the centre.  These considerations lead to the Fujimura sets
\begin{align*}
B_{0,4} &\backslash \{ (0,0,4), (0,4,0), (4,0,0) \}\\
B_{0,5} &\backslash \{ (0,4,1), (0,5,0), (4,0,1), (5,0,0) \}\\
B_{0,6} &\backslash \{ (0,1,5), (0,5,1), (1,0,5), (0,1,5), (1,5,0), (5,1,0) \}
\end{align*}
which by \eqref{cn3} gives the lower bounds
$$ c_{4,3} \geq 52; c_{5,3} \geq 150; c_{6,3} \geq 450.$$
Thus we have established all the lower bounds needed for Theorem \ref{dhj-upper}.

One can of course continue this process by hand, for instance the set
$$ B_{0,7} \backslash \{(0,1,6),(1,0,6),(0,5,2),(5,0,2),(1,5,1),(5,1,1),(1,6,0),(6,1,0) \}$$
gives the lower bound $c_{7,3} \geq 1302$, which we tentatively conjecture to be the correct bound. 

A simplification was found when $n$ is a multiple of $3$.  Observe that for $n=6$, the sets excluded from $B_{0,6}$ are all permutations of $(0,1,5)$.  So the remaining sets are all the permutations of $(1,2,3)$ and $(0,2,4)$.  In the same way, sets for $n=9$, $12$ and $15$ can be described as:
\begin{itemize}
\item $n=9$: $(2,3,4),(1,3,5),(0,4,5)$ and permutations;
\item $n=12$: $(3,4,5),(2,4,6),(1,5,6),(0,2,10),(0,5,7)$ and permutations;
\item $n=15$: $(4,5,6),(3,5,7),(2,6,7),(1,3,11),(1,6,8),(0,4,11),(0,7,8)$ and permutations.
\end{itemize}

When $n$ is not a multiple of $3$, say $n=3m-1$ or $n=3m-2$, one first finds a solution for $n=3m$.  Then for $n=3m-1$, one restricts the first digit of the $3m$ sequence to equal $1$.  This leaves exactly one-third as many points for $3m-1$ as for $3m$.  For $n=3m-1$, one restricts the first two digits of the $3m$ sequence to be $12$.  This leaves roughly one-ninth as many points for $3m-2$ as for $3m$.

An integer program\footnote{Details of the integer programming used in this paper can be found at the page {\tt Integer.tex} at \cite{polywiki}.} was solved to obtain the maximum lower bound one could establish from \eqref{cn3}.  The results for $1 \leq n \leq 20$ are displayed in Figure \ref{nlow}.
More complete data, including the list of optimisers, can be found at \cite{markstrom}.

\begin{figure}[tb]
\centerline{
\begin{tabular}{|ll|ll|}
\hline
$n$ & lower bound & $n$ & lower bound \\
\hline
1 & 2 &11&	96338\\
2 & 6 & 12&	287892\\
3 &	18 & 13&	854139\\
4 &	52 & 14&	2537821\\
5 &	150& 15&	7528835\\
6 &	450& 16&	22517082\\
7 &	1302& 17&	66944301\\
8 &	3780&18&	198629224\\
9 &	11340&19&	593911730\\
10&	32864& 20&	1766894722\\
\hline
\end{tabular}}
\caption{Lower bounds for $c_n$ obtained by the $A_B$ construction.}
\label{nlow}
\end{figure}

For medium values of $n$, in particular for integers $21 \leq n \leq 999$ that are a multiple of $3$, $n=3m$, the best general lower bound for $c_{n,3}$ was found by applying \eqref{cn3} to the following Fujimura set construction.
It is convenient to write $[a,b,c]$ for the point $(m+a,m+b,m+c)$, together with its permutations, with the convention that $[a,b,c]$ is empty if these points do not lie in $\Delta_{n,3}$.  Then a Fujimura set can be constructed by taking the following groups of points:

\begin{enumerate}
\item The thirteen groups of points
$$[-7,-3,+10], [-7, 0,+7], [-7,+3,+4], [-6,-4,+10], [-6,-1,+7], [-6,+2,+4]$$
$$[-5,-1,+6],[-5,+2,+3],[-4,-2,+6], [-4,+1,+3], [-3,+1,+2], [-2,0,+2], [-1,0,+1];$$
\item The four families of groups
$$[-8-y-2x,-6+y-2x,14+4x], [-8-y-2x,-3+y-2x,11+4x], $$
$$[-8-y-2x,x+y,8+x], [-8-2x,3+x,5+x]$$
for $x \geq 0$ and $y=0,1$.
\end{enumerate}

Numerical computation shows that this construction gives a line-free set in $[3]^n$ of density approximately $2.7 \sqrt{\frac{\log n}{n}}$ for $n \leq 1000$; for instance, when $n=99$, it gives a line-free set of density at least $1/3$.  Some additional constructions of this type can be found at the page {\tt Upper and lower bounds} at \cite{polywiki}.

However, the bounds in Theorem \ref{dhj-lower}, which we now prove, are asymptotically superior to these constructions.

\begin{proof}[Proof of Theorem \ref {dhj-lower}] 
Let $M$ be the circulant matrix with first row $(1,2,\ldots,k-1)$, second row $(k-1,1,2,\dots,k-2)$, and so on. Note that $M$ has nonzero determinant by well-known properties\footnote{For instance, if we let $A_i$ denote the $i^{th}$ row, we see that $(A_1-A_2)+(A_{i+1}-A_i)$ is of the form $(0,\ldots,0,-k+1,0,\ldots,0,k-1)$, and so the row space spans all the vectors whose coordinates sum to zero; but the first row has a non-zero coordinate sum, so the rows in fact span the whole space.}  of circulant matrices, see e.g. \cite[Theorem 3]{kra}.

Let $S$ be a subset of the interval $[-\sqrt {n}/2, \sqrt {n}/2)$ that contains no nonconstant arithmetic progressions of length $k$, and let $B\subset\Delta_{n, k}$ be the set 
    \[ B := \{(n-\sum_{i=1}^{k-1} a_i ,a_1,a_2,\dots, a_{k-1}) : 
            (a_1,\dots,a_{k-1})= c + \det(M) M^{-1} s , s\in S^{k-1}\},\]
where $c$ is the $k-1$-dimensional vector, all of whose entries are equal to $\lfloor n/k \rfloor$. 
The map $(m,a_1,\dots,a_{k-1}) \mapsto M (a_1,\dots,a_{k-1})$ takes simplices in $\Delta_{n,k}$ to nonconstant arithmetic progressions in ${\mathbb Z}^{k-1}$, and takes $B$ to $\{M c +\det(M) \, s \colon s \in S^{k-1}\}$, which is a set containing no nonconstant arithmetic progressions. Thus, $B$ is a Fujimura set and so does not contain any combinatorial lines. 

If all of $a_1,\ldots,a_k$ are within $C_1\sqrt{n}$ of $n/k$, then $|\Gamma_{\vec{a}}| \geq C k^n/n^{(k-1)/2}$ (where $C$ depends on $C_1$) by the central limit theorem. By our choice of $S$ and applying~\eqref{cn3} (or more precisely, the obvious generalisation of \eqref{cn3} to other values of $k$), we obtain 
     $$ c_ {n, k}\geq C k^n/n^{(k-1)/2} |S|^{k-1} = C k^n \left( \frac{|S|}{\sqrt{n}} \right)^{k-1}. $$
One can take $S$ to have cardinality $r_ k (\sqrt {n}) $, which from the results of O'Bryant~\cite {obryant} satisfies (for all sufficiently large $n$, some $C>0$, and $\ell$ the largest integer satisfying $k> 2^{\ell-1}$) 
     $$ \frac{r_k (\sqrt{n})}{\sqrt{n}} \geq C  (\log n)^{1/(2\ell)}\exp_2 (-\ell 2^{(\ell-1)/2-1/\ell} \sqrt[\ell]{\log_2 n}),$$
which completes the proof.
\end {proof}

\section{Upper bounds for the $k=3$ density Hales-Jewett problem}\label{dhj-upper-sec}

To finish the proof of Theorem \ref{dhj-upper} we need to supply the indicated upper bounds for $c_{n,3}$ for $n=0,\ldots,6$.  

It is clear that $c_{0,3} = 1$ and $c_{1,3} = 2$.  By subdividing a line-free set into three parallel slices we obtain the bound
$$ c_{n+1,3} \leq 3 c_{n,3}$$
for all $n$.  This is already enough to get the correct upper bounds $c_{2,3} \leq 6$ and $c_{3,3} \leq 18$, and also allows us to deduce the upper bound $c_{6,3} \leq 450$ from $c_{5,3} \leq 150$.  So the remaining tasks are to establish the upper bounds 
\begin{equation}\label{c43}
c_{4,3} \leq 52
\end{equation}
and
\begin{equation}\label{c53}
c_{5,3} \leq 150.
\end{equation}

In order to establish \eqref{c53}, we will rely on \eqref{c43}, together with a classification of those line-free sets in $[3]^4$ of size close to the maximal number $52$.  Similarly, to establish \eqref{c43}, we will need the bound $c_{3,3} \leq 18$, together with a classification of those line-sets in $[3]^3$ of size close to the maximal number $18$.  Finally, to achieve the latter aim one needs to classify the line-free subsets of $[3]^2$ with exactly $c_{2,3}=6$ elements.

\subsection{$n=2$}

We begin with the $n=2$ theory.

\begin{lemma}[$n=2$ extremals]\label{2d-ext}  There are exactly four line-free subsets of $[3]^2$ of cardinality $6$:
\begin{itemize}
\item The set $x := A_{B_{2,2}} = \{12, 13, 21, 22, 31, 33\}$;
\item The set $y := A_{B_{2,1}} = \{11, 12, 21, 23, 32, 33\}$;
\item The set $z := A_{B_{2,0}} = \{11, 13, 22, 23, 31, 32\}$;
\item The set $w := \{12, 13, 21, 23, 31, 32\}$. 
\end{itemize}
\end{lemma}

\begin{proof}  A line-free subset of $[3]^2$ must have exactly two elements in every row and column.  The claim then follows by brute force search.
\end{proof}

\subsection{$n=3$}

Now we turn to the $n=3$ theory.  We can slice $[3]^3$ as the union of three slices $1**$, $2**$, $3**$, each of which are identified with $[3]^2$ in the obvious manner.  Thus every subset $A$ in $[3]^3$ can be viewed as three subsets $A_1, A_2, A_3$ of $[3]^2$ stacked together; if $A$ is line-free then $A_1,A_2,A_3$ are necessarily line-free, but the converse is not true.  We write $A = A_1 A_2 A_3$, thus for instance $xyz$ is the set
$$ xyz = \{112,113,121,122,131,133\} \cup \{211,212,221,223,232,233\} \cup \{311,313,322,323,331,332\}.$$
Observe that
$$ A_{B_{0,3}} = xyz; \quad A_{B_{1,3}} = yzx; \quad A_{B_{2,3}} = zxy.$$

\begin{lemma}[$n=3$ extremals]\label{Lemma1} The only $18$-element line-free subset of $[3]^3$ is $xyz$.  The only $17$-element line-free subsets of $[3]^3$ are formed by removing a point from $xyz$, or by removing either $111$, $222$, or $333$ from $yzx$ or $zxy$.
\end{lemma}

\begin{proof} We prove the second claim. As $17 = 6 + 6 + 5$, and $c_{2,3} = 6$, at least two of the slices of a $17$-element line-free set must be from $x$, $y$, $z$, $w$, with the third slice having $5$ points. If two of the slices are identical, the last slice must lie in the complement and thus has at most $3$ points, a contradiction. If one of the slices is a $w$, then the $5$-point slice consists of the complement of the other two slices and thus contains a diagonal, contradiction. By symmetry we may now assume that two of the slices are $x$ and $y$, which force the last slice to be $z$ with one point removed. Now one sees that the slices must be in the order $xyz$, $yzx$, or $zxy$, because any other combination has too many lines that need to be removed. The sets $yzx$, $zxy$ contain the diagonal $\{111,222,333\}$ and so one additional point needs to be removed.

The first claim follows by a similar argument to the second. 
\end{proof}

\subsection{$n=4$}

Now we turn to the $n=4$ theory.  

\begin{lemma}  $c_{4,3} \leq 52$.
\end{lemma}

\begin{proof} Let $A$ be a line-free set in $[3]^4$, and split $A=A_1A_2A_3$ for $A_1,A_2,A_3 \in [3]^3$ as in the $n=3$ theory.  If at least two of the slices $A_1,A_2,A_3$ are of cardinality $18$, then by Lemma \ref{Lemma1} they are of the form $xyz$, and so the third slice then lies in the complement and has at most six points, leading to an inferior bound of $18+18+6 = 42$.  Thus at most one of the slices can have cardinality $18$, leading to the bound $18+17+17=52$.
\end{proof}

Now we classify extremisers.  Observe that we have the following (equivalent) $52$-point line-free sets, which were implicitly constructed in the previous section;

\begin{itemize}
\item $E_0 := A_{B_{0,4}}\backslash\{1111,2222\}$;
\item $E_1 := A_{B_{1,4}}\backslash\{2222,3333\}$;
\item $E_2 := A_{B_{2,4}}\backslash\{1111,3333\}$.
\end{itemize}

\begin{lemma}\label{Lemma2}\ 
\begin{itemize}
\item The only $52$-element line-free sets in $[3]^4$ are $E_0$, $E_1$, $E_2$.
\item The only $51$-element line-free sets in $[3]^4$ are formed by removing a point from $E_0$, $E_1$ or $E_2$.
\item The only $50$-element line-free sets in $[3]^4$ are formed by removing two points from $E_0$, $E_1$ or $E_2$ OR are equal to one of the three permutations of the set $X := \Gamma_{3,1,0} \cup \Gamma_{3,0,1} \cup \Gamma_{2,2,0} \cup \Gamma_{2,0,2} \cup \Gamma_{1,1,2} \cup \Gamma_{1,2,1} \cup \Gamma_{0,2,2}$.
\end{itemize} 
\end{lemma}

\begin{proof} We will just prove the third claim, which is the hardest; the first two claims follow from the same argument (and can in fact be deduced directly from the third claim).

It suffices to show that every $50$-point line-free set is either contained in the $54$-point set $A_{B_{j,4}}$ for some $j=0,1,2$, or is some permutation of the set $X$. Indeed, if a $50$-point line-free set is contained in, say, $A_{B_{0,4}}$, then it cannot contain $2222$, since otherwise it must omit one point from each of the four pairs formed from $\{2333, 2111\}$ by permuting the indices, and must also omit one of $\{1111, 1222, 1333\}$, leading to at most $49$ points in all; similarly, it cannot contain $1111$, and so omits the entire diagonal $\{1111,2222,3333\}$, with two more points to be omitted. By symmetry we see the same argument works when $A_{B_{0,4}}$ is replaced by one of the other $A_{B_{j,4}}$.

Next, observe that every three-dimensional slice of a line-free set can have at most $c_{3,3} = 18$ points; thus when one partitions a $50$-point line-free set into three such slices, it must divide either as $18+16+16$, $18+17+15$, $17+17+16$, or some permutation of these.
Suppose that we can slice the set into two slices of $17$ points and one slice of $16$ points. By the various symmetries, we may assume that the $1***$ slice and $2***$ slices have $17$ points, and the $3***$ slice has $16$ points. By Lemma \ref{Lemma1}, the $1$-slice is $\{1\}\times D_{3,j}$ with one point removed, and the $2$-slice is $\{2\}\times D_{3,k}$ with one point removed, for some $j,k \in \{0,1,2\}$.
If $j=k$, then the $1$-slice and $2$-slice have at least $15$ points in common, so the $3$-slice can have at most $27 - 15 = 12$ points, a contradiction. If $jk = 01$, $12$, or $20$, then observe that from Lemma \ref{Lemma1} the $*1**$, $*2**$, $*3**$ slices cannot equal a $17$-point or $18$-point line-free set, so each have at most $16$ points, leading to only $48$ points in all, a contradiction. Thus we must have $jk = 10$, $21$, or $02$.

First suppose that $jk=02$. Then by Lemma \ref{Lemma1}, the $2***$ slice contains the nine points formed from $\{2211, 2322, 2331\}$ and permuting the last three indices, while the $1***$ slice contains at least eight of the nine points formed from $\{1211, 1322, 1311\}$ and permuting the last three indices. Thus the $3***$ slice can contain at most one of the nine points formed from $\{3211, 3322, 3311\}$ and permuting the last three indices. If it does contain one of these points, say $3211$, then it must omit one point from each of the four pairs $\{3222, 3233\}$, $\{3212, 3213\}$, $\{3221, 3231\}$, $\{3111, 3311\}$, leading to at most $15$ points on this slice, a contradiction. So the $3***$ slice must omit all nine points, and is therefore contained in $\{3\}\times D_{3,1}$, and so the $50$-point set is contained in $D_{4,1}$, and we are done by the discussion at the beginning of the proof.

The case $jk=10$ is similar to the $jk=02$ case (indeed one can get from one case to the other by swapping the $1$ and $2$ indices). Now suppose instead that $jk=12$. Then by Lemma \ref{Lemma1}, the $1***$ slice contains the six points from permuting the last three indices of $1123$, and similarly the $2***$ slice contains the six points from permuting the last three indices of $2123$. Thus the $3***$ slice must avoid all six points formed by permuting the last three indices of $3123$. Similarly, as $1133$ lies in the $1***$ slice and $2233$ lies in the $2***$ slice, $3333$ must be avoided in the $3***$ slice.

Now we claim that $3111$ must be avoided also; for if $3111$ was in the set, then one point from each of the six pairs formed from $\{3311, 3211\}$, $\{3331, 3221\}$ and permuting the last three indices must lie outside the $3***$ slice, which reduces the size of that slice to at most $27 - 6 - 1 - 6 = 14$, which is too small. Similarly, $3222$ must be avoided, which puts the $3***$ slice inside $\{3\}\times D_3$ and then places the $50$-point set inside $D_4$, and we are done by the discussion at the beginning of the proof.

We have handled the case in which at least one of the slicings of the $50$-point set is of the form $50=17+17+16$. The only remaining case is when all slicings of the $50$-point set are of the form $18+16+16$ or $18+17+15$ (or a permutation thereof). So each slicing includes an $18$-point slice.  By the symmetries of the situation, we may assume that the $1***$ slice has $18$ points, and thus by Lemma \ref{Lemma1} takes the form $\{1\}\times D_3$. Inspecting the $*1**$, $*2**$, $*3**$ slices, we then see (from Lemma \ref{Lemma1}) that only the $*1**$ slice can have $18$ points; since we are assuming that this slicing is some permutation of $18+17+15$ or $18+16+16$, we conclude that the $*1**$ slice must have exactly $18$ points, and is thus described precisely by Lemma \ref{Lemma1}. Similarly for the $**1*$ and $***1$ slices. Indeed, by Lemma \ref{Lemma1}, we see that the 50-point set must agree exactly with $D_{4,1}$ on any of these slices. In particular, there are exactly six points of the 50-point set in the remaining portion $\{2,3\}^4$ of the cube.

Suppose that $3333$ was in the set; then since all permutations of $3311$, $3331$ are known to lie in the set, then $3322$, $3332$ must lie outside the set. Also, as $1222$ lies in the set, at least one of $2222$, $3222$ lie outside the set. This leaves only $5$ points in $\{2,3\}^4$, a contradiction. Thus $3333$ lies outside the set; similarly $2222$ lies outside the set.

Let $a$ be the number of points in the $50$-point set which are some permutation of $2233$, thus $0\leq a\leq 6$. If $a=0$ then the set lies in $D_{4,1}$ and we are done. If $a=6$ then the set is exactly $X$ and we are done. Now suppose $a=1$. By symmetry we may assume that $2233$ lies in the set. Then (since $2133, 1233, 2231, 2213$ are known to lie in the set) $2333, 3233, 2223, 2232$ lie outside the set, which leaves at most $5$ points inside $\{2,3\}^4$, a contradiction.  A similar argument holds if $a=2,3$.

The remaining case is when $a=4,5$. Then one of the three pairs $\{2233, 3322\}$, $\{2323, 3232\}$, $\{2332, 3223\}$ lie in the set. By symmetry we may assume that $\{2233, 3322\}$ lie in the set. Then by arguing as before we see that all eight points formed by permuting $2333$ or $3222$ lie outside the set, leading to at most 5 points inside $\{2,3\}^4$, a contradiction. 
\end{proof}

\subsection{$n=5$}

Finally, we turn to the $n=5$ theory.  Our goal is to show that $c_{5,3} \leq 150$.  Accordingly, suppose for contradiction that we can find a line-free subset $A$ of $[3]^5$ of cardinality $|A|=151$.  We will now prove a series of facts about $A$ which will eventually give the desired contradiction.

\begin{lemma} $A$ is not contained inside $A_{B_{j,5}}$ for any $j=0,1,2$.  
\end{lemma}

\begin{proof} Suppose for contradiction that $A \subset A_{B_{j,5}}$ for some $j$.  By symmetry we may take $j=0$.  The set $A_{B_{0,5}}$ has $162$ points. By looking at the triplets $\{10000, 11110, 12220\}$ and cyclic permutations we must lose $5$ points; similarly from the triplets $\{20000,22220, 21110\}$ and cyclic permutations. Finally from $\{11000,11111,11222\}$ and $\{22000,22222,22111\}$ we lose two more points.   Since $162-5-5-2=150$, we obtain the desired contradiction.
\end{proof}

Observe that every slice of $A$ contains at most $c_{4,3}=52$ points, and hence every slice of $A$ contains at least $151-52-52=47$ points.

\begin{lemma} \label{NoTwo51} $A$ cannot have two parallel $[3]{}^4$ slices, each of which contain at least $51$ points.
\end{lemma}

\begin{proof} Suppose not that $A$ has two parallel $[3]^4$ slices. By symmetry, we may assume that the $1****$ and $2****$ slices have at least $51$ points.  Meanwhile, the $3****$ slice has at least $47$ points as discussed above.

By Lemma \ref{Lemma2}, the $1****$ slice takes the form $\{1\}\times D_{4,j}$  for some $j = 0,1,2$ with the diagonal $\{11111,12222,13333\}$ and possibly one more point removed, and similarly the $2****$ slice takes the form $\{2\}\times D_{4,k}$ for some $k = 0,1,2$ with the diagonal $\{21111,22222,23333\}$ and possibly one more point removed.

Suppose first that $j=k$. Then the $1$-slice and $2$-slice have at least $50$ points in common, leaving at most $31$ points for the $3$-slice, a contradiction. Next, suppose that $jk=01$. Then observe that the $*i***$ slice cannot look like any of the configurations in Lemma \ref{Lemma2} and so must have at most $50$ points for $i=1,2,3$, leading to $150$ points in all, a contradiction. Similarly if $jk=12$ or $20$. Thus we must have $jk$ equal to $10$, $21$, or $02$.

Let's suppose first that $jk=10$. The first slice then is equal to $\{1\}\times D_{4,1}$ with the diagonal and possibly one more point removed, while the second slice is equal to $\{2\}\times D_{4,0}$ with the diagonal and possibly one more point removed. Superimposing these slices, we thus see that the third slice is contained in $\{3\}\times D_{4,2}$ except possibly for two additional points, together with the one point $32222$ of the diagonal that lies outside of $\{3\}\times D_{4,2}$.

The lines $x12xx, x13xx$ (plus permutations of the last four digits) must each contain one point outside the set. The first two slices can only absorb two of these, and so at least $14$ of the $16$ points formed by permuting the last four digits of $31233$, $31333$ must lie outside the set. These points all lie in $\{3\}\times D_{4,2}$, and so the $3****$ slice can have at most $| D_{4,2}| - 14 + 3 = 43$ points, a contradiction.

The case $jk=02$ is similar to the case $jk=10$ (indeed one can obtain one from the other by swapping $1$ and $2$). Now we turn to the case $jk=21$. Arguing as before we see that the third slice is contained in $\{3\}\times D_4$ except possibly for two points, together with $33333$.

If $33333$ was in the set, then each of the lines $xx333$, $xxx33$ (and permutations of the last four digits) must have a point missing from the first two slices, which cannot be absorbed by the two points we are permitted to remove; thus $33333$ is not in the set. For similar reasons, $33331$ is not in the set, as can be seen by looking at $xxx31$ and permutations of the last four digits. Indeed, any string containing four threes does not lie in the set; this means that at least $8$ points are missing from $\{3\}\times D_{4}$, leaving only at most $46$ points inside that set. Furthermore, any point in the $3****$ slice outside of $\{3\}\times D_{4}$ can only be created by removing a point from the first two slices, so the total cardinality is at most $46 + 52 + 52 = 150$, a contradiction.
\end{proof}

\begin{remark} This already gives the bound $c_{5,3} \leq 52+50+50=152$, but of course we wish to do better than this.
\end{remark}

\begin{lemma}\label{151-49} $A$ has a slice $j****$ with $j=1,2,3$ that has at most $49$ points. 
\end{lemma}

\begin{proof} Suppose not, thus all three slices of $A$ has at least $50$ points.
Using earlier notation, we split subsets of $[3]{}^4$ into nine subsets of $[3]{}^2$. So we think of $x,y,z,a,b$ and $c$ as subsets of a square. By Lemma \ref{Lemma2}, each slice is one of the following:
\begin{itemize}
\item $E_0 = y'zx,zx'y,xyz$ (with one or two points removed)
\item $E_1 = xyz,yz'x,zxy'$ (with one or two points removed)
\item $E_2 = z'xy,xyz,yzx'$ (with one or two points removed)
\item $X = xyz,ybw,zwc$
\item $Y = axw,xyz,wzc$
\item $Z = awx,wby,xyz$
\end{itemize}
where $a$, $b$ and $c$ have four points each: $a = \{2,3\}^2$, $b = \{1,3\}^2$ and $c = \{1,2\}^2$. $x'$, $y'$ and $z'$ are subsets of $x$, $y$ and $z$ respectively, and have five points each.

Suppose all three slices are subsets of $E_{j_1}, E_{j_2}, E_{j_3}$ respectively for some $j_1,j_2,j_3 \in \{0,1,2\}$, $E_1$, or $E_2$. We can remove at most five points from the full set $E_{j_1} \uplus E_{j_2} \uplus E_{j_3}$. Consider columns $2,3,4,6,7,8$. At most two of these columns contain $xyz$, so one point must be removed from the other four. This uses up all but one of the removals. So the slices must be $E_2$, $E_1$, $E_0$ or a cyclic permutation of that. Then the cube, which contains the first square of slice $1$; the fifth square of slice $2$; and the ninth square of slice $3$, contains three copies of the same square. It takes more than one point removed to remove all lines from that cube. So we can't have all three slices subsets of $E_j$.

Suppose one slice is $X$, $Y$ or $Z$, and two others are subsets of $E_j$. We can remove at most three points from the two $E_j$. By symmetry, suppose one slice is $X$. Consider columns $2$, $3$, $4$ and $7$. They must be cyclic permutations of $x$, $y$, $z$, and two of them are not $xyz$, so must lose a point. Columns $6$ and $8$ must both lose a point, and we only have $150$ points left. So if one slice is $X$, $Y$ or $Z$, the full set contains a line.

Suppose two slices are from $X$, $Y$ and $Z$, and the other is a subset of $E_j$. By symmetry, suppose two slices are $X$ and $Y$. Columns $3$, $6$, $7$ and $8$ all contain $w$, and therefore at most $16$ points each. Columns $1$, $5$ and $9$ contain $a$, $b$, or $c$, and therefore at most $16$ points. So the total number of points is at most $7 \times 16+2 \times 18 = 148 < 151$, a contradiction.
\end{proof}

This, combined with Lemma \ref{NoTwo51}, gives

\begin{corollary}\label{525049} Any three parallel slices of $A$ must have cardinality $52, 50, 49$ (or a permutation thereof).  
\end{corollary}

Note that this argument already gives the bound $c_{5,3} \leq 151$.

\begin{lemma} \label{151NoX} No slice $j****$ of $A$ is of the form $X$, where $X$ was defined in Lemma \ref{Lemma2}. 
\end{lemma}

\begin{proof} Suppose one slice is $X$; then by the previous discussion one of the parallel slices has $52$ points and is thus of the form $E_j$ for some $j=0,1,2$, by Lemma \ref{Lemma2}.

Suppose that $X$ is the first slice $1****$.  We have  $X = xyz\ ybw \ zwc$. Label the other rows with letters from the alphabet, thus
$$ A = \begin{pmatrix} xyz & ybw & zwc \\ mno & pqr & stu \\def & ghi & jkl 
\end{pmatrix}$$
Reslice the array into a left nine, middle nine and right nine. One of these squares contains $52$ points, and it can only be the left nine. One of its three columns contains $18$ points, and it can only be its left-hand column, $xmd$. So $m=y$ and $d=z$. But none of the $E_j$ begins with $y$ or $z$, which is a contradiction. So $X$ is not in the first row.

So $X$ is in the second or third row. By symmetry, suppose it is in the second row, so that $A$ has the following shape:
$$ A = \begin{pmatrix} def & ghi & jkl \\xyz & ybw & zwc \\ mno & pqr & stu \end{pmatrix}$$
Again, the left-hand nine must contain $52$ points, so it is $E_2$. Now, to get $52$ points in any row, the first row must be $E_2$. Then the only way to have $50$ points in the middle or right-hand nine is if the middle nine is $X$:
$$ A = \begin{pmatrix} z'xy & xyz & yzx' \\xyz & ybw & zwc \\ yzx' & zwc & stu \end{pmatrix}$$
In the seventh column, $s$ contains $5$ points and in the eighth column, $t$ contains $4$ points. The final row can now contain at most $48$ points, contradicting Corollary \ref{525049}.

A similar argument is possible if $X$ is in the third row; or if $X$ is replaced by $Y$ or $Z$.  Thus, given any decomposition of $A$ into three parallel slices, one slice is a $52$-point set $E_j$ and another slice is $50$ points contained in $E_k$. \end{proof}

Now we can obtain the desired contradiction:

\begin{lemma} There is no $151$-point line-free set $A \subset [3]^5$. \end{lemma}

\begin{proof} Assume by symmetry that the first row contains $52$ points and the second row contains $50$.
If $E_1$ is in the first row, then the second row must be contained in $E_0$:
$$ A = \begin{pmatrix} xyz & yz'x & zxy' \\ y'zx & zx'y & xyz \\ def & ghi & jkl\end{pmatrix}$$
But then none of the left nine, middle nine or right nine can contain $52$ points, which contradicts Corollary \ref{525049}.
Suppose the first row is $E_0$. Then the second row is contained in $E_2$, otherwise the cubes formed from the nine columns of the diagram would need to remove too many points:
$$ A = \begin{pmatrix} y'zx & zx'y & xyz \\ z'xy & xyz & yzx' \\ def & ghi & jkl \end{pmatrix}.$$
But then neither the left nine, middle nine nor right nine contain $52$ points.
So the first row contains $E_2$, and the second row is contained in $E_1$. Two points may be removed from the second row of this diagram:
$$ A = \begin{pmatrix} z'xy & xyz & yzx' \\ xyz & yz'x & zxy' \\ def & ghi & jkl \end{pmatrix}.$$
Slice it into the left nine, middle nine and right nine. Two of them are contained in $E_j$ so at least two of $def$, $ghi$, and $jkl$ are contained in the corresponding slice of $E_0$. Slice along a different axis, and at least two of $dgj$, $ehk$, $fil$ are contained in the corresponding slice of $E_0$. So eight of the nine squares in the bottom row are contained in the corresponding square of $E_0$. Indeed, slice along other axes, and all points except one are contained within $E_0$. This point is the intersection of all the $49$-point slices.
So, if there is a $151$-point solution, then after removal of the specified point, there is a $150$-point solution, within $D_{5,j}$, whose slices in each direction are $52+50+48$.
$$ A = \begin{pmatrix} z'xy & xyz & yzx' \\ xyz & yz'x & zxy' \\ y'zx & zx'y & xyz \end{pmatrix}$$
One point must be lost from columns $3$, $6$, $7$ and $8$, and four more from the major diagonal $z'z'z$. That leaves $148$ points instead of $150$.
So the $150$-point solution does not exist with $52+50+48$ slices; so the $151$ point solution does not exist. 
\end{proof}

This establishes that $c_{5,3} \leq 150$, and thus $c_{6,3} \leq 3c_{5,3} \leq 450$.
\section{Lower bounds for the Moser problem}\label{moser-lower-sec}

Just as for the density Hales-Jewett problem, we found that Gamma sets $\Gamma_{a,b,c}$ were useful in providing large lower bounds for the Moser problem.  This is despite the fact that the symmetries of the cube do not respect Gamma sets.

Observe that if $B \subset \Delta_n$, then the set $A_B := \bigcup_{\vec a \in B} \Gamma_{a,b,c}$ is a Moser set as long as $B$ does not contain any ``isosceles triangles'' $(a+r,b,c+s), (a+s,b,c+r), (a,b+r+s,c)$ for any $r,s \geq 0$ not both zero; in particular, $B$ cannot contain any ``vertical line segments'' $(a+r,b,c+r), (a,b+2r,c)$.  An example of such a set is provided by selecting $0 \leq i \leq n-3$ and letting $B$ consist of the triples $(a, n-i, i-a)$ when $a \neq 3 \mod 3$, $(a,n-i-1,i+1-a)$ when $a \neq 1 \mod 3$, $(a,n-i-2,i+2-a)$ when $a=0 \mod 3$, and $(a,n-i-3,i+3-a)$ when $a=2 \mod 3$.  Asymptotically, this set includes about two thirds of the spheres $S_{n,i}$, $S_{n,i+1}$ and one third of the spheres $S_{n,i+2}, S_{n,i+3}$ and (setting $i$ close to $n/3$) gives a lower bound 
\begin{equation}\label{cpn3}
c'_{n,3} \geq (C-o(1)) 3^n / \sqrt{n} 
\end{equation} 
with $C = 2 \times \sqrt{\frac{9}{4\pi}}$.  This lower bound is the asymptotic limit of our methods; see Proposition \ref{props} below.

An integer program was solved to obtain the optimal lower bounds achievable by the $A_B$ construction (using \eqref{cn3}, of course).  The results for $1 \leq n \leq 20$ are displayed in Figure \ref{nlow-moser}. More complete data, including the list of optimisers, can be found at \cite{markstrom}.

\begin{figure}[tb]
\centerline{
\begin{tabular}{|ll|ll|}
\hline
$n$ & lower bound & $n$ & lower bound \\
\hline
1 & 2 &11&	71766\\
2 & 6 & 12&	212423\\
3 &	16 & 13&	614875\\
4 &	43 & 14&	1794212\\
5 &	122& 15&	5321796\\
6 &	353& 16&	15455256\\
7 &	1017& 17&	45345052\\
8 &	2902&18&	134438520\\
9 &	8622&19&	391796798\\
10&	24786& 20&	1153402148\\
\hline
\end{tabular}}
\caption{Lower bounds for $c'_{n,3}$ obtained by the $A_B$ construction.}
\label{nlow-moser}
\end{figure}

\begin{figure}[tb]
\centerline{\includegraphics{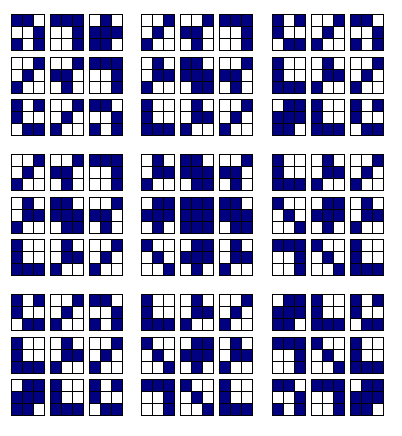}}
\caption{One of the examples of $353$-point sets in $[3]^6$ (elements of the set being indicated by white squares).  This example was generated by a genetic algorithm.}
\label{moser353-fig}
\end{figure}

Unfortunately, any method based purely on the $A_B$ construction cannot do asymptotically better than the previous constructions:

\begin{proposition}\label{props}  Let $B \subset \Delta_n$ be such that $A_B$ is a Moser set.  Then $|A_B| \leq (2 \sqrt{\frac{9}{4\pi}} + o(1)) \frac{3^n}{\sqrt{n}}$.
\end{proposition}

\begin{proof}  By the previous discussion, $B$ cannot contain any pair of the form $(a,b+2r,c), (a+r,b,c+r)$ with $r>0$.  In other words, for any $-n \leq h \leq n$, $B$ can contain at most one triple $(a,b,c)$ with $c-a=h$.  From this and \eqref{cn3}, we see that
$$ |A_B| \leq \sum_{h=-n}^n \max_{(a,b,c) \in \Delta_n: c-a=h} \frac{n!}{a! b! c!}.$$
From the Chernoff inequality (or the Stirling formula computation below) we see that $\frac{n!}{a! b! c!} \leq \frac{1}{n^{10}} 3^n$ unless $a,b,c = n/3 + O( n^{1/2} \log^{1/2} n )$, so we may restrict to this regime, which also forces $h = O( n^{1/2}\log^{1/2} n)$.  If we write $a = n/3 + \alpha$, $b = n/3 + \beta$, $c = n/3+\gamma$ and apply Stirling's formula $n! = (1+o(1)) \sqrt{2\pi n} n^n e^{-n}$, we obtain 
$$ \frac{n!}{a! b! c!} = (1+o(1)) \frac{3^{3/2}}{2\pi n} 3^n \exp( - (\frac{n}{3}+\alpha) \log (1 + \frac{3\alpha}{n} ) - (\frac{n}{3}+\beta) \log (1 + \frac{3\beta}{n} ) - (\frac{n}{3}+\gamma) \log (1 + \frac{3\gamma}{n} ) ).$$
From Taylor expansion one has
$$ -(\frac{n}{3}+\alpha) \log (1 + \frac{3\alpha}{n} ) = -\alpha - \frac{3}{2} \frac{\alpha^2}{n} + o(1)$$
and similarly for $\beta,\gamma$; since $\alpha+\beta+\gamma=0$, we conclude that
$$ \frac{n!}{a! b! c!} = (1+o(1)) \frac{3^{3/2}}{2\pi n} 3^n \exp( - \frac{3}{2n} (\alpha^2+\beta^2+\gamma^2) ).$$
If $c-a=h$, then $\alpha^2+\beta^2+\gamma^2 = \frac{3\beta^2}{2} + \frac{h^2}{2}$.  Thus we see that
$$ \max_{(a,b,c) \in \Delta_n: c-a=h} \frac{n!}{a! b! c!} \leq (1+o(1)) \frac{3^{3/2}}{2\pi n} 3^n \exp( - \frac{3}{4n} h^2 ).$$
Using the integral test, we thus have
$$ |A_B| \leq (1+o(1)) \frac{3^{3/2}}{2\pi n} 3^n  \int_\R \exp( - \frac{3}{4n} x^2 )\ dx.$$
Since $\int_\R \exp( - \frac{3}{4n} x^2 )\ dx = \sqrt{\frac{4\pi n}{3}}$, we obtain the claim.
\end{proof}

Actually it is possible to improve upon these bounds by a slight amount.  Observe that if $B$ is a maximiser for the right-hand side of \eqref{cn3} (subject to $B$ not containing isosceles triangles), then any triple $(a,b,c)$ not in $B$ must be the vertex of a (possibly degenerate) isosceles triangle with the other vertices in $B$.  If this triangle is non-degenerate, or if $(a,b,c)$ is the upper vertex of a degenerate isosceles triangle, then no point from $\Gamma_{a,b,c}$ can be added to $A_B$ without creating a geometric line. However, if $(a,b,c) = (a'+r,b',c'+r)$ is only the lower vertex of a degenerate isosceles triangle $(a'+r,b',c'+r), (a',b'+2r,c')$, then one can add any subset of $\Gamma_{a,b,c}$ to $A_B$ and still have a Moser set as long as no pair of elements in that subset is separated by Hamming distance $2r$.  For instance, in the $n=5$ case, we can start with the 122-point set built from 
$$B = \{ (0 0 5),(0 2 3),(1 1 3 ),(1 2 2 ),(2 2 1 ),(3 1 1 ),(3 2 0 ),(5 0 0))\}$$
and add a point each from $(1 0 4)$ and $(4 0 1)$.
This gives an example of the maximal, 124-point solution.  Again, in the $n=10$ case, the set
\begin{align*}
B &= \{(0 0 10),(0 2 8 ),(0 3 7 ),(0 4 6 ),(1 4 5 ),(2 1 7 ),(2 3 5 ), 
(3 2 5 ),(3 3 4 ),(3 4 3 ),(4 4 2 ),(5 1 4 ), \\
&\quad (5 3 2 ),(6 2 2 ), 
(6 3 1 ),(6 4 0 ),(8 1 1 ),(9 0 1 ),(9 1 0 ) \}
\end{align*}

generates the lower bound $c'_{10,3} \geq 24786$ given above (and, up to reflection $a \leftrightarrow c$, is the only such set that does so); but by adding the following twelve elements from $\Gamma_{5,0,5}$ one can increase the lower bound slightly to $24798$: $1111133333$, $1111313333$,
$1113113333$, $1133331113$, $1133331131$, $1133331311$, $3311333111$,
$3313133111$, $3313313111$, $3331111133$, $3331111313$, $3331111331$

A more general form goes with the $B$ set described at the start of this section.  Include points from $\Gamma_{(a,n-i-4,i+4-a)}$ when $a=1 (\operatorname{mod}\ 3)$, subject to no two points being included if they differ by the interchange of a $1$ and a $3$.  Each of these Gamma sets is the feet of a degenerate isosceles triangle with vertex $\Gamma_{(a-1,n-i-2,a+3-a)}$.  

\begin{lemma} If $A$ is a subset of $\Gamma_{(a,b,c)}$ such that no two points of $A$ differ by the interchange of a $1$ and a $3$, then $|A| \leq |\Gamma_{a,b,c}|/(1+\max(a,c))$.
\end{lemma}
\begin{proof}
Say that two points of $\Gamma_{a,b,c}$ are neighbours if they differ by the exchange of a $1$ and a $3$.  Each point of $A$ has $ac$ neighbours, none of which are in $A$. Each point of $\Gamma_{(a,b,c)}\backslash A$ has $ac$ neighbours, but only $\min(a,c)$ of them may be in $A$. So for each point of $A$ there are on average $ac/\min(a,c) = \max(a,c)$ points not in $A$. So the proportion of points of $\Gamma_{(a,b,c)}$ that are in $A$ is at most one in $1+\max(a,c)$.
\end{proof}

The proportion of extra points for each of the cells $\Gamma_{(a,n-i-4,i+4-a)}$ is no more than $2/(i+6)$.  Only one cell in three is included from the  $b=n-i-4$ layer, so we expect no more than $\binom{n}{i+4}2^{i+5}/(3i+18)$ new points, all from $S_{n,i+4}$.  One can also find extra points from $S_{n,i+5}$ and higher spheres.

Earlier solutions may also give insight into the problem. Clearly we have 
$c'_{0,3}=1$ and $c'_{1,3}=2$, so we focus on the case $n \ge 2$.
The first lower bounds may be due to Koml\'{o}s \cite{komlos}, who observed 
that the sphere $S_{i,n}$ of elements with exactly $n-i$ 2 entries
(see Section \ref{notation-sec} for definition),
is a Moser set, so that 
\begin{equation}\label{cin}
c'_{n,3}\geq \vert S_{i,n}\vert
\end{equation}
holds for all $i$. Choosing $i=\lfloor \frac{2n}{3}\rfloor$  and 
applying Stirling's formula, we see that this lower bound takes the form \eqref{cpn3} with $C := \sqrt{\frac{9}{4\pi}}$.
In particular $c'_{3,3} \geq 12, c'_{4,3}\geq 24, c'_{5,3}\geq 80, 
c'_{6,3}\geq 240$.

These values can be improved by studying combinations of several spheres or
semispheres or applying elementary results from coding theory.

Observe that if $\{w(1),w(2),w(3)\}$ is a geometric line in $[3]^n$, 
then $w(1), w(3)$ both lie in the same sphere $S_{i,n}$, 
and that $w(2)$ lies in a lower 
sphere $S_{i-r,n}$ for some $1 \leq r \leq i \leq n$. Furthermore, 
$w(1)$ and $w(3)$ are separated by Hamming distance $r$.

As a consequence, we see that $S_{i-1,n} \cup S_{i,n}^e$ (or $S_{i-1,n} 
\cup S_{i,n}^o$) is a Moser set for any $1 \leq i \leq n$, since any two 
distinct elements $S_{i,n}^e$ are separated by a Hamming distance of at 
least two. 
(Recall Section \ref{notation-sec} for definitions),
This leads to the lower bound 
\begin{equation}\label{cn3-low}
c'_{n,3} \geq \binom{n}{i-1} 
2^{i-1} + \binom{n}{i} 2^{i-1} = \binom{n+1}{i} 2^{i-1}.
\end{equation}
It is not 
hard to see that $\binom{n+1}{i+1} 2^{i} > \binom{n+1}{i} 2^{i-1}$ if 
and only if $3i < 2n+1$, and so this lower bound is maximised when $i = 
\lfloor \frac{2n+1}{3} \rfloor$ for $n \geq 2$, giving the formula 
\eqref{binom}. This leads to the lower bounds $$ c'_{2,3} \geq 6; 
c'_{3,3} \geq 16; c'_{4,3} \geq 40; c'_{5,3} \geq 120; c'_{6,3} \geq 
336$$ which gives the right lower bounds for $n=2,3$, but is slightly 
off for $n=4,5$.  Asymptotically, Stirling's formula and \eqref{cn3-low} then
give the lower bound \eqref{cpn3} with $C = \frac{3}{2} \times \sqrt{\frac{9}{4\pi}}$, which is asymptotically $50\%$ better than the bound \eqref{cin}.

The work of Chv\'{a}tal \cite{chvatal1} 
already contained a refinement of this idea
which we here translate into the usual notation of coding theory:
Let $A(n,d)$ denote the size of the largest binary code of length $n$
and minimal distance $d$.

Then 
\begin{equation}\label{cnchvatal}
c'_{n,3}\geq \max_k \left( \sum_{j=0}^k \binom{n}{j} A(n-j, k-j+1)\right).
\end{equation}

The following values of $A(n,d)$ for small $n,d$ are known, see \cite{Brower}:
{\tiny{
\[
\begin{array}{llllllll}
A(1,1)=2&&&&&&&\\
A(2,1)=4& A(2,2)=2&&&&&&\\
A(3,1)=8&A(3,2)=4&A(3,3)=2&&&&&\\
A(4,1)=16&A(4,2)=8& A(4,3)=2& A(4,4)=2&&&&\\
A(5,1)=32&A(5,2)=16& A(5,3)=4& A(5,4)=2&A(5,5)=2&&&\\
A(6,1)=64&A(6,2)=32& A(6,3)=8& A(6,4)=4&A(6,5)=2&A(6,6)=2&&\\
A(7,1)=128&A(7,2)=64& A(7,3)=16& A(7,4)=8&A(7,5)=2&A(7,6)=2&A(7,7)=2&\\
A(8,1)=256&A(8,2)=128& A(8,3)=20& A(8,4)=16&A(8,5)=4&A(8,6)=2
&A(8,7)=2&A(8,8)=2\\
A(9,1)=512&A(9,2)=256& A(9,3)=40& A(9,4)=20&A(9,5)=6&A(9,6)=4
&A(9,7)=2&A(9,8)=2\\
A(10,1)=1024&A(10,2)=512& A(10,3)=72& A(10,4)=40&A(10,5)=12&A(10,6)=6
&A(10,7)=2&A(10,8)=2\\
A(11,1)=2048&A(11,2)=1024& A(11,3)=144& A(11,4)=72&A(11,5)=24&A(11,6)=12
&A(11,7)=2&A(11,8)=2\\
A(12,1)=4096&A(12,2)=2048& A(12,3)=256& A(12,4)=144&A(12,5)=32&A(12,6)=24
&A(12,7)=4&A(12,8)=2\\
A(13,1)=8192&A(13,2)=4096& A(13,3)=512& A(13,4)=256&A(13,5)=64&A(12,6)=32
&A(13,7)=8&A(13,8)=4\\
\end{array}
\]
}}
In addition, one has the general identities $A(n,1)=2^n, A(n,2)=2^{n-1}, A(n-1,2e-1)=A(n,2e)$, and $A(n,d)=2$, 
if $d>\frac{2n}{3}$.

Inserting this data into \eqref{cnchvatal} for $k=2$ we obtain the lower bounds
\[
\begin{array}{llll}
c'_{4,3}&\geq &\binom{4}{0}A(4,3)+\binom{4}{1}A(3,2)+\binom{4}{2}A(2,1)
=1\cdot 2+4 \cdot 4+6\cdot 4&=42\\
c'_{5,3}&\geq &\binom{5}{0}A(5,3)+\binom{5}{1}A(4,2)+\binom{5}{2}A(3,1)
=1\cdot 4+5 \cdot 8+10\cdot 8&=124\\
c'_{6,3}&\geq &\binom{6}{0}A(6,3)+\binom{6}{1}A(5,2)+\binom{6}{2}A(4,1)
=1\cdot 8+6 \cdot 16+15\cdot 16&=344.
\end{array}
\]
Similarly, with $k=3$ we obtain
\[
\begin{array}{llll}
c'_{7,3}&\geq& \binom{7}{0}A(7,4)+\binom{7}{1}A(6,3)+\binom{7}{2}A(5,2)
+ \binom{7}{3}A(4,1)&=960.\\
c'_{8,3}&\geq &\binom{8}{0}A(8,4)+\binom{8}{1}A(7,3)+\binom{8}{2}A(6,2)
+ \binom{8}{3}A(5,1)&=2832\\
c'_{9,3}&\geq & \binom{9}{0}A(9,4)+\binom{9}{1}A(8,3)+\binom{9}{2}A(7,2)
+ \binom{9}{3}A(6,1)&=7880
\end{array}\]
and for $k=4$ we have
\[
\begin{array}{llll}
c'_{10,3}&\geq &\binom{10}{0}A(10,5)+\binom{10}{1}A(9,4)+\binom{10}{2}A(8,3)
+ \binom{10}{3}A(7,2)+\binom{10}{4}A(6,1)&=22232\\
c'_{11,3}&\geq &\binom{11}{0}A(11,5)+\binom{11}{1}A(10,4)+\binom{11}{2}A(9,3)
+ \binom{11}{3}A(8,2)+\binom{11}{4}A(7,1)&=66024\\
c'_{12,3}&\geq &\binom{12}{0}A(12,5)+\binom{12}{1}A(11,4)+\binom{12}{2}A(10,3)
+ \binom{12}{3}A(9,2)+\binom{12}{4}A(8,1)&=188688.\\
\end{array}\]
and for $k=5$ we have
\[ c'_{13,3}\geq 539168.\]

It should be pointed out that these bounds are even numbers, so that 
$c'_{4,3}=43$ shows that one cannot generally expect this lower bound 
to give the optimum.

The maximum value  appears to occur for $k=\lfloor\frac{n+2}{3}\rfloor$, but even after optimising in these parameters and using explicit bounds on $A(n,d)$ we were unable to improve upon the constant $C = 2 \times \sqrt{\frac{9}{4\pi}}$
for \eqref{cpn3} arising from previously discussed constructions.  Using the singleton bound $A(n,d)\leq 2^{n-d+1}$ Chv\'{a}tal   \cite{chvatal1} proved that the expression on the right hand side of
\eqref{cnchvatal} is also 
$O\left( \frac{3^n}{\sqrt{n}}\right)$, so that the refinement described 
above gains a constant factor over the initial construction only.
 
For $n=4$ the above does not yet give the exact value.
The value $c'_{4,3}=43$ was first proven by Chandra \cite{chandra}.
A uniform way of describing examples for the optimum values of 
$c'_{4,3}=43$ and $c'_{5,3}=124$ is as follows.

Let us consider the sets $$ A := S_{i-1,n} 
\cup S_{i,n}^e \cup A'$$ where $A' \subset S_{i+1,n}$ has the property 
that any two elements in $A'$ are separated by a Hamming distance of at 
least three, or have a Hamming distance of exactly one but their 
midpoint lies in $S_{i,n}^o$. By the previous discussion we see that 
this is a Moser set, and we have the lower bound 
\begin{equation}\label{cnn}
c'_{n,3} \geq \binom{n+1}{i} 2^{i-1} + |A'|.
\end{equation} This gives some improved lower bounds for $c'_{n,3}$:

\begin{itemize} \item By taking $n=4$, $i=3$, and $A' = \{ 1111, 3331, 
3333\}$, we obtain $c'_{4,3} \geq 43$; \item By taking $n=5$, $i=4$, and 
$A' = \{ 11111, 11333, 33311, 33331 \}$, we obtain $c'_{5,3} \geq 124$. 
\item By taking $n=6$, $i=5$, and $A' = \{ 111111, 111113, 111331, 
111333, 331111, 331113\}$, we obtain $c'_{6,3} \geq 342$. \end{itemize}

This gives the lower bounds in Theorem \ref{moser} up to $n=5$, but 
the bound for $n=6$ is inferior to the lower bound $c'_{6,3}\geq 344$ given above.

\subsection{Higher $k$ values}

We now consider lower bounds for $c'_{n,k}$ for some values of $k$ larger than $3$.  Here we will see some further connections between the Moser problem and the density Hales-Jewett problem.

For $k=4$, we have the lower bounds $c'_{n,4} \ge \binom{n}{n/2}2^n$.  To see this, observe that the set of points with $a$ $1$s,$b$ $2$s,$c$ $3$s and $d$ $4$s, where $a+d$ has the constant value $n/2$, does not form geometric lines because points at the ends of a geometric line have more $a$ or $d$ values than points in the middle of the line.

The following lower bound is asymptotically twice as large.  Take all points with $a$ $1$s, $b$ $2$s, $c$ $3$s and $d$ $4$s, for which:

\begin{itemize}
\item Either $a+d = q$ or $q-1$, $a$ and $b$ have the same parity; or
\item $a+d = q-2$ or $q-3$, $a$ and $b$ have opposite parity.
\end{itemize}

This includes half the points of four adjacent layers, and therefore may include $(1+o(1))\binom{n}{n/2}2^{n+1}$ points.

We also have a lower bound for $c'_{n,5}$ similar to Theorem \ref{dhj-lower}, namely $c'_{n,5} = 5^{n-O(\sqrt{\log n})}$. Consider points with $a$ $1$s, $b$ $2$s, $c$ $3$s, $d$ $4$s and $e$ $5$s.  For each point, take the value $a+e+2(b+d)+3c$.  The first three points in any geometric line give values that form an arithmetic progression of length three.  

Select a set of integers with no arithmetic progression of length $3$.  Select all points whose value belongs to that sequence; there will be no geometric line among those points.  By the Behrend construction\cite{behrend}, it is possible to choose these points with density $\exp{-O(\sqrt{\log n})}$.

For $k=6$, we observe that the asymptotic $c'_{n,6} = o(6^n)$ would imply the $k=3$ density Hales-Jewett theorem $c_{n,3}=o(3^n)$. Indeed, any $k=3$ combinatorial line-free set can be ``doubled up'' into a $k=6$ geometric line-free set of the same density by pulling back the set from the map that maps $1, 2, 3, 4, 5, 6$ to $1, 2, 3, 3, 2, 1$ respectively; note that this map sends $k=6$ geometric lines to $k=3$ combinatorial lines.  So $c'_{n,6} \geq 2^n c_{n,3}$, and more generally, $c'_{n,2k} \geq 2^n c_{n,k}$.

\section{Upper bounds for the $k=3$ Moser problem in small dimensions}\label{moser-upper-sec}

In this section we finish the proof of Theorem \ref{moser} by obtaining the upper bounds on  $c'_{n,3}$ for $n \leq 6$.

\subsection{Statistics, densities and slices}

Our analysis will revolve around various \emph{statistics} of Moser sets $A \subset [3]^n$, their associated \emph{densities}, and the behavior of such statistics and densities with respect to the operation of passing from the cube $[3]^n$ to various \emph{slices} of that cube.

\begin{definition}[Statistics and densities]  Let $A \subset [3]^n$ be a set.  For any $0 \leq i \leq n$, set $a_i(A) := |A \cap S_{n-i,n}|$; thus we have
$$ 0 \leq a_i(A) \leq |S_{n-i,n}| = \binom{n}{i} 2^{n-i}$$
for $0 \leq i \leq n$ and
$$ a_0(A) + \ldots + a_n(A) = |A|.$$
We refer to the vector $(a_0(A),\ldots,a_n(A))$ as the \emph{statistics} of $A$.  We define the $i^{th}$ \emph{density} $\alpha_i(A)$ to be the quantity
$$ \alpha_i(A) := \frac{a_i(A) }{\binom{n}{i} 2^{n-i}},$$
thus $0 \leq \alpha_i(A) \leq 1$ and
$$ |A| = \sum_{i=0}^n \binom{n}{i} 2^{n-i} \alpha_i(A).$$
\end{definition}

\begin{example}\label{2mos}  Let $n=2$ and $A$ be the Moser set $A := \{ 12, 13, 21, 23, 31, 32 \}$.  Then the statistics $(a_0(A), a_1(A), a_2(A))$ of $A$ are $(2,4,0)$, and the densities $(\alpha_0(A), \alpha_1(A), \alpha_2(A))$ are $(\frac{1}{2}, 1, 0)$.  
\end{example}

When working with small values of $n$, it will be convenient to write $a(A)$, $b(A)$, $c(A)$, etc. for $a_0(A)$, $a_1(A)$, $a_2(A)$, etc., and similarly write $\alpha(A), \beta(A), \gamma(A)$, etc. for $\alpha_0(A)$, $\alpha_1(A)$, $\alpha_2(A)$, etc.  Thus for instance in Example \ref{2mos} we have $b(A) = 4$ and $\alpha(A) = \frac{1}{2}$.

\begin{definition}[Subspace statistics and densities]  If $V$ is a $k$-dimensional geometric subspace of $[3]^n$, then we have a map $\phi_V: [3]^k \to [3]^n$ from the $k$-dimensional cube to the $n$-dimensional cube.  If $A \subset [3]^n$ is a set and $0 \leq i \leq k$, we write $a_i(V,A)$ for $a_i(\phi_V^{-1}(A))$ and $\alpha_i(V,A)$ for $\alpha_i(\phi_V^{-1}(A))$.  If the set $A$ is clear from context, we abbreviate $a_i(V,A)$ as $a_i(V)$ and $\alpha_i(V,A)$ as $\alpha_i(V)$.
\end{definition}

Recall from Section \ref{notation-sec} that the cube $[3]^n$ can be subdivided into three \emph{slices} in $n$ different ways, and each slice is an $n-1$-dimensional subspace.  For instance, $[3]^3$ can be partitioned into $1**$, $2**$, $3**$.  We call a slice a \emph{centre slice} if the fixed coordinate is $2$ and a \emph{side slice} if it is $1$ or $3$.
 
\begin{example}  We continue Example \ref{2mos}.  Then the statistics of the side slice $1*$ are $(a(1*),b(1*)) = (1,1)$, while the statistics of the centre slice $2*$ are $(a(2*),b(2*))=(2,0)$.  The corresponding densities are $(\alpha(1*),\beta(1*)) = (1/2,1)$ and $(\alpha(2*),\beta(2*))=(1,0)$.
\end{example}

A simple double counting argument gives the following useful identity:

\begin{lemma}[Double counting identity]\label{dci}  Let $A \subset [3]^n$ and $0 \leq i \leq n-1$.  Then we have
$$ \frac{1}{n-i-1} \sum_{V \hbox{ a side slice}} a_{i+1}(V) = \frac{1}{i+1} \sum_{W \hbox{ a centre slice}} a_i(W) = a_{i+1}(A)$$
where $V$ ranges over the $2n$ side slices of $[3]^n$, and $W$ ranges over the $n$ centre slices.  In other words, the average value of $\alpha_{i+1}(V)$ for side slices $V$ equals the average value of $\alpha_i(W)$ for centre slices $W$, which is in turn equal to $\alpha_{i+1}(A)$.
\end{lemma}

Indeed, this lemma follows from the observation that every string in $A \cap S_{n-i-1,n}$ belongs to $i+1$ centre slices $W$ (and contributes to $a_i(W)$) and to $n-i-1$ side slices $V$ (and contributes to $a_{i+1}(V)$).  One can also view this lemma probabilistically, as the assertion that there are three equivalent ways to generate a random string of length $n$:

\begin{itemize}
\item Pick a side slice $V$ at random, and randomly fill in the wildcards in such a way that $i+1$ of the wildcards are $2$'s (i.e. using an element of $S_{n-i-2,n-1}$).
\item Pick a centre slice $V$ at random, and randomly fill in the wildcards in such a way that $i$ of the wildcards are $2$'s (i.e. using an element of $S_{n-i-1,n-1}$).
\item Randomly choose an element of $S_{n-i-1,n}$.
\end{itemize}

\begin{example} We continue Example \ref{2mos}.  The average value of $\beta$ for side slices is equal to the average value of $\alpha$ for centre slices, which is equal to $\beta(A) = 1$.
\end{example}

Another very useful fact (essentially due to \cite{chvatal2}) is that linear inequalities for statistics of Moser sets at one dimension propagate to linear inequalities in higher dimensions:

\begin{lemma}[Propagation lemma]\label{prop}  Let $n \geq 1$ be an integer.  Suppose one has a linear inequality of the form
\begin{equation}\label{alphav}
 \sum_{i=0}^n v_i \alpha_i(A) \leq s
\end{equation}
for all Moser sets $A \subset [3]^n$ and some real numbers $v_0,\ldots,v_n,s$.  Then we also have the linear inequality
$$ \sum_{i=0}^n v_i \alpha_{qi+r}(A) \leq s$$
whenever $q \geq 1$, $r \geq 0$, $N \geq nq+r$ are integers and $A \subset [3]^N$ is a Moser set.
\end{lemma}

\begin{proof}  We run a probabilistic argument (one could of course also use a double counting argument instead).  Let $n,v_0,\ldots,v_n,s,q,r,N,A$ be as in the lemma.  Let $V$ be a random $n$-dimensional geometric subspace of $[3]^N$, created in the following fashion:
\begin{itemize}
\item Pick $n$ wildcards $x_1,\ldots,x_n$ to run independently from $1$ to $3$.  We also introduce dual wildcards $\overline{x_1},\ldots,\overline{x_n}$; each $\overline{x_j}$ will take the value $4-x_j$.
\item We randomly subdivide the $N$ coordinates into $n$ groups of $q$ coordinates, plus a remaining group of $N-nq$ ``fixed'' coordinates.
\item For each coordinate in the $j^{th}$ group of $q$ coordinates for $1 \leq j \leq n$, we randomly assign either a $x_j$ or $\overline{x_j}$.
\item For each coordinate in the $N-nq$ fixed coordinates, we randomly assign a digit $1,2,3$, but condition on the event that exactly $r$ of the digits are equal to $2$ (i.e. we use a random element of $S_{N-nq-r,N-nq}$).
\item Let $V$ be the subspace created by allowing $x_1,\ldots,x_n$ to run independently from $1$ to $3$, and $\overline{x_j}$ to take the value $4-x_j$.
\end{itemize}

For instance, if $n=2, q=2, r=1, N=6$, then a typical subspace $V$ generated in this fashion is
$$ 2x_1\overline{x_2}3x_2x_1 = \{ 213311, 212321, 211331, 223312, 222322, 221332, 233313, 232323, 231333\}.$$
Observe from that the following two ways to generate a random element of $[3]^N$ are equivalent:

\begin{itemize}
\item Pick $V$ randomly as above, and then assign $(x_1,\ldots,x_n)$ randomly from $S_{n-i,n}$.  Assign $4-x_j$ to $\overline{x_j}$ for all $1 \leq j \leq n$.
\item Pick a random string in $S_{N-qi-r,N}$.
\end{itemize}

Indeed, both random variables are invariant under the symmetries of the cube, and both random variables always pick out strings in $S_{N-qi-r,N}$, and the claim follows.  As a consequence, we see that the expectation of $\alpha_i(V)$ (as $V$ ranges over the recipe described above) is equal to $\alpha_{qi+r}(A)$.  On the other hand, from \eqref{alphav} we have
$$  \sum_{i=0}^n v_i \alpha_i(V) \leq s$$
for all such $V$; taking expectations over $V$, we obtain the claim.
\end{proof}

In view of Lemma \ref{prop}, it is of interest to locate linear inequalities relating the densities $\alpha_i(A)$, or (equivalently) the statistics $a_i(A)$.  For this, it is convenient to introduce the following notation.

\begin{definition}  Let $n \geq 1$ be an integer.  
\begin{itemize}
\item A vector $(a_0,\ldots,a_n)$ of non-negative integers is \emph{feasible} if it is the statistics of some Moser set $A$.
\item A feasible vector $(a_0,\ldots,a_n)$ is \emph{Pareto-optimal} if there is no other feasible vector $(b_0,\ldots,b_n) \neq (a_0,\ldots,a_n)$ such that $b_i \geq a_i$ for all $0 \leq i \leq n$.
\item A Pareto-optimal vector $(a_0,\ldots,a_n)$ is \emph{extremal} if it is not a non-trivial convex linear combination of other Pareto-optimal vectors.
\end{itemize}
\end{definition}

To establish a linear inequality of the form \eqref{alphav} with the $v_i$ non-negative, it suffices to test the inequality against densities associated to extremal vectors of statistics.  (There is no point considering linear inequalities with negative coefficients $v_i$, since one always has the freedom to reduce a density $\alpha_i(A)$ of a Moser set $A$ to zero, simply by removing all elements of $A$ with exactly $i$ $2$'s.)

We will classify exactly the Pareto-optimal and extremal vectors for $n \leq 3$, which by Lemma \ref{prop} will lead to useful linear inequalities for $n \geq 4$.  Using a computer, we have also located a partial list of Pareto-optimal and extremal vectors for $n=4$, which are also useful for the $n=5$ and $n=6$ theory.

\subsection{Up to three dimensions}

We now establish Theorem \ref{moser} for $n \leq 3$, and establish some auxiliary inequalities which will be of use in higher dimensions.

The case $n=0$ is trivial. When $n=1$, it is clear that $c'_{1,3} = 2$, and furthermore that the Pareto-optimal statistics are $(2,0)$ and $(1,1)$, which are both extremal.  This leads to the linear inequality
$$ 2\alpha(A) + \beta(A) \leq 2$$
for all Moser sets $A \subset [3]^1$, which by Lemma \ref{prop} implies that
\begin{equation}\label{alpha-1}
2\alpha_r(A) + \alpha_{r+q}(A) \leq 2
\end{equation}
whenever $r \geq 0$, $q \geq 1$, $n \geq q+r$, and $A \subset [3]^n$ is a Moser set.

For $n=2$, we see by partitioning $[3]^2$ into three slices that $c'_{2,3} \leq 3 c'_{1,3} = 6$, and so (by the lower bounds in the previous section) $c'_{2,3} = 6$.  Writing $(a,b,c) = (a(A),b(A),c(A)) = (4\alpha(A), 4\beta(A), \gamma(A))$, the inequalities \eqref{alpha-1} become
\begin{equation}\label{abc}
a + 2c \leq 4; b+2c \leq 4; 2a+b <= 8.
\end{equation}

\begin{lemma}  When $n=2$, the Pareto-optimal statistics are $(4,0,0), (3,2,0), (2,4,0), (2,2,1)$. In particular, the extremal statistics are $(4,0,0), (2,4,0), (2,2,1)$.
\end{lemma}

\begin{proof}  One easily checks that all the statistics listed above are feasible.
Consider the statistics $(a,b,c)$ of a Moser set $A \subset [3]^2$.  $c$ is either equal to $0$ or $1$.  If $c=1$, then \eqref{abc} implies that $a,b \leq 2$, so the only Pareto-optimal statistic here is $(2,2,1)$.  When instead $c=0$, the inequalities \eqref{abc} can easily imply the Pareto-optimality of $(4,0,0), (3,2,0), (2,4,0)$.
\end{proof}

From this lemma we see that we obtain a new inequality $2a+b+2c \leq 8$.  Converting this back to densities and using Lemma \ref{prop}, we conclude that
\begin{equation}\label{alpha-2}
4\alpha_r(A) + 2\alpha_{r+q}(A) + \alpha_{r+2q} \leq 4
\end{equation}
whenever $r \geq 0$, $q \geq 1$, $n \geq q+2r$, and $A \subset [3]^n$ is a Moser set.

The line-free subsets of $[3]^2$ can be easily exhausted by computer search; it turns out that there are $230$ such sets.

Now we look at three dimensions.  Writing $(a,b,c,d)$ for the statistics of a Moser set $A \subset [3]^n$ (which thus range between $(0,0,0,0)$ and $(8,12,6,1)$), the inequalities \eqref{alpha-1} imply in particular that
\begin{equation}\label{abc-3d}
a+4d \leq 8; b+6d \leq 12; c+3d \leq 6; 3a+2c \leq 24; b+c \leq 12
\end{equation}
while \eqref{alpha-2} implies that
\begin{equation}\label{abcd-3d}
3a+b+c \leq 24; b+c+3d \leq 12.
\end{equation}
Summing the inequalities $b+c \leq 12, 3a+b+c \leq 24, b+c+3d \leq 12$ yields
$$ 3(a+b+c+d) \leq 48$$
and hence $|A| = a+b+c+d \leq 16$; comparing this with the lower bounds of the preceding section we obtain $c'_{3,3} = 16$ as required.  (This argument is essentially identical to the one in \cite{chvatal2}).

We have the following useful computation:

\begin{lemma}[3D Pareto-optimals]\label{paretop} When $n=3$, the Pareto-optimal statistics are $$(3,6,3,1),(4,4,3,1),(4,6,2,1),(2,6,6,0),(3,6,5,0),(4,4,5,0),(3,7,4,0),(4,6,4,0),$$ 
$$ (3,9,3,0),(4,7,3,0),(5,4,3,0),(4,9,2,0),(5,6,2,0),(6,3,2,0),(3,10,1,0),(5,7,1,0),$$
$$ (6,4,1,0),(4,12,0,0),(5,9,0,0),(6,6,0,0),(7,3,0,0),(8,0,0,0).$$  
In particular, the extremal statistics are 
$$(3,6,3,1),(4,4,3,1),(4,6,2,1),(2,6,6,0),(4,4,5,0),(4,6,4,0),(4,12,0,0),(8,0,0,0).$$
\end{lemma}

\begin{proof} This can be established by a brute-force search over the $2^{27} \approx 1.3 \times 10^8$ different subsets of $[3]^3$.  Actually, one can perform a much faster search than this.  Firstly, as noted earlier, there are only $230$ line-free subsets of $[3]^2$, so one could search over $230^3 \approx 1.2 \times 10^7$ configurations instead.  Secondly, by symmetry we may assume (after enumerating the $230$ sets in a suitable fashion) that the first slice $A \cap 1**$ has an index less than or equal to the third $A \cap 3**$, leading to $\binom{231}{2} \times 230 \approx 6 \times 10^6$ configurations instead.  Finally, using the first and third slice one can quickly determine which elements of the second slice $2**$ are prohibited from $A$.  There are $2^9 = 512$ possible choices for the prohibited set in $2**$.  By crosschecking these against the list of $230$ line-free sets one can compute the Pareto-optimal statistics for the second slices inside the prohibited set (the lists of such statistics turns out to length at most $23$).  Storing these statistics in a lookup table, and then running over all choices of the first and third slice (using symmetry), one now has to perform $O( 512 \times 230 ) + O( \binom{231}{2} \times 23) \approx O( 10^6 )$ computations, which is quite a feasible computation.

One could in principle reduce the computations even further, by a factor of up to $8$, by using the symmetry group $D_4$ of the square $[3]^2$ to reduce the number of cases one needs to consider, but we did not implement this.

A computer-free proof of this lemma can be found at the page {\tt Human proof of the 3D Pareto-optimal Moser statistics} at \cite{polywiki}.
\end{proof}

\begin{remark} A similar computation revealed that the total number of line-free subsets of $[3]^3$ was $3813884$.  With respect to the $2^3 \times 3!=48$-element group of geometric symmetries of $[3]^3$, these sets partitioned into $83158$ equivalence classes:
$$
3813884 = 76066 \times 48+6527 \times 24+51 \times 16+338 \times 12 +109 \times 8+41 \times 6+13 \times 4 +5 \times 3+3 \times 2+5 \times 1.
$$
\end{remark}

Lemma \ref{paretop} yields the following new inequalities:
\begin{align*}
2a+b+2c+4d &\leq 22 \\
3a+2b+3c+6d &\leq 36 \\
7a+2b+4c+8d &\leq 56 \\
6a+2b+3c+6d &\leq 48 \\
a+2c+4d &\leq 14 \\
5a+4c+8d &\leq 40.
\end{align*}

Applying Lemma \ref{prop}, we obtain new inequalities:
\begin{align}
8\alpha_r(A)+ 6\alpha_{r+q}(A) + 6\alpha_{r+2q}(A) + 2\alpha_{r+3q}(A) &\leq 11 \label{eleven}\\
4\alpha_r(A)+4\alpha_{r+q}(A)+3\alpha_{r+2q}(A)+\alpha_{r+3q}(A) &\leq 6\label{six}\\
7\alpha_r(A)+3\alpha_{r+q}(A)+3\alpha_{r+2q}(A)+\alpha_{r+3q}(A) &\leq 7\nonumber\\
8\alpha_r(A)+3\alpha_{r+q}(A)+3\alpha_{r+2q}(A)+\alpha_{r+3q}(A) &\leq 8\label{eight}\\
4\alpha_{r+q}(A)+2\alpha_{r+2q}(A)+\alpha_{r+3q}(A) &\leq 4\nonumber\\
4\alpha_r(A)+6\alpha_{r+2q}(A)+2\alpha_{r+3q}(A) &\leq 7\nonumber\\
5\alpha_r(A)+3\alpha_{r+2q}(A)+\alpha_{r+3q}(A) &\leq 5\nonumber 
\end{align}
whenever $r \geq 0$, $q \geq 1$, $n \geq r+3q$, and Moser sets $A \subset [3]^n$.

We also note some further corollaries of Lemma \ref{paretop}:

\begin{corollary}[Statistics of large 3D Moser sets]\label{paretop2}  Let $(a,b,c,d)$ be the statistics of a Moser set $A$ in $[3]^3$.  Then $|A| = a+b+c+d \leq 16$.  Furthermore:
\begin{itemize}
\item If $|A|=16$, then $(a,b,c,d) = (4,12,0,0)$.
\item If $|A|=15$, then $(a,b,c,d) = (4,11,0,0)$ or $(3,12,0,0)$.
\item If $|A| \geq 14$, then $b \geq 6$ and $d=0$.
\item If $|A| = 13$ and $d=1$, then $(a,b,c,d) = (4,6,2,1)$ or $(3,6,3,1)$.
\end{itemize}
\end{corollary}

\subsection{Four dimensions}

Now we establish the bound $c'_{4,3}=43$.  Let $A$ be a Moser set in $[3]^4$, with attendant statistics $(a,b,c,d,e)$, which range between $(0,0,0,0,0)$ and $(16,32,24,8,1)$.  In view of the lower bounds, our task here is to establish the upper bound $a+b+c+d+e \leq 43$.

The linear inequalities already established just barely fail to achieve this bound, but we can obtain the upper bound $a+b+c+d+e \leq 44$ as follows.
First suppose that $e=1$; then from the inequalities \eqref{alpha-1} (or by considering lines passing through $2222$) we see that $a \leq 8, b \leq 16, c \leq 12, d \leq 4$ and hence $a+b+c+d+e \leq 41$, so we may assume that $e=0$.

From Lemma \ref{dci}, we see that $a+b+c+d+e$ is now equal to the sum of $a(V)/4+b(V)/3+c(V)/2+d(V)$, where $V$ ranges over all side slices of $[3]^4$.  But from Lemma \ref{paretop} we see that $a(V)/4+b(V)/3+c(V)/2+d(V)$ is at most $\frac{11}{4}$, with equality occurring only when $(a(V),b(V),c(V),d(V))=(2,6,6,0)$.  This gives the upper bound $a+b+c+d+e \leq 44$.

The above argument shows that $a+b+c+d+e=44$ can only occur if $e=0$ and if $(a(V),b(V),c(V),d(V))=(2,6,6,0)$ for all side slices $V$.  Applying Lemma \ref{paretop} again this implies $(a,b,c,d,e)=(4,16,24,0,0)$.  But then $A$ contains all of the sphere $S_{2,4}$, which implies that the four-element set $A \cap S_{4,4}$ cannot contain a pair of strings which differ in exactly two positions (as their midpoint would then lie in $S_{2,4}$, contradicting the hypothesis that $A$ is a Moser set).  

Recall that we may partition $S_{4,4} = S_{4,4}^e \cup S_{4,4}^o$, where 
$$S_{4,4}^e := \{ 1111, 1133, 1313, 3113, 1331, 3131, 3311, 3333\}$$
is the strings in $S_{4,4}$ with an even number of $1$'s, and 
$$S_{4,4}^o := \{ 1113, 1131, 1311, 3111, 1333, 3133, 3313, 3331\}$$
are the strings in $S_{4,4}$ with an odd number.  Observe that any two distinct elements in $S_{4,4}^e$ differ in exactly two positions unless they are antipodal.  Thus $A \cap S_{4,4}^e$ has size at most two, with equality only when $A \cap S_{4,4}^e$ consists of an antipodal pair.  Similarly for $A \cap S_{4,4}^o$.  Thus $A$ must consist of two antipodal pairs, one from $S_{4,4}^e$ and one from $S_{4,4}^o$.

By the symmetries of the cube we may assume without loss of generality that these pairs are $\{ 1111, 3333\}$ and $\{1113,3331\}$ respectively.  But as $A$ is a Moser set, $A$ must now exclude the strings $1112$ and $3332$.  These two strings form two corners of the eight-element set
$$ ***2 \cap S_{3,4} = \{ 1112, 1132, 1312, 3112, 1332, 3132, 3312, 3332 \}.$$
Any pair of points in this set which are ``adjacent'' in the sense that they differ by exactly one entry cannot both lie in $A$, as their midpoint would then lie in $S_{3,4}$, and so $A$ can contain at most four elements from this set, with equality only if $A$ contains all the points in $***2 \cap S_{3,4}$ of the same parity (either all the elements with an even number of $3$s, or all the elements with an odd number of $3$s).  But because the two corners removed from this set have the opposite parity (one has an even number of $1$s and one has an odd number), we see in fact that $A$ can contain at most $3$ points from this set.  Meanwhile, the same arguments give that $A$ contains at most four points from $**2* \cap S_{3,4}$, $*2** \cap S_{3,4}$, and $2*** \cap S_{3,4}$.  Summing we see that $b = |A \cap S_{3,4}| \leq 3+4+4+4=15$, a contradiction.  Thus we have $c'_{4,3}=43$ as claimed.

We have the following four-dimensional version of Lemma \ref{paretop}:

\begin{lemma}[4D Pareto-optimals]\label{paretop-4} When $n=4$, the Pareto-optimal statistics are given by the table in Figure \ref{4d-pareto}.
\end{lemma}

\begin{figure}
{\tiny
\begin{tabular}{|ll|}
\hline
$a$ & $(b,c,d,e)$ \\ 
\hline
$3$ & $(16, 24)$ \\
$4$& $(14, 19, 2)$,
$(15, 24)$,
$(16, 8, 4, 1)$,
$(16, 14, 4)$,
$(16, 23)$,
$(17, 21)$,
$(18, 19)$\\
$5$& $(12, 12, 4, 1)$,
$(12, 13, 6)$,
$(12, 15, 5)$,
$(12, 19, 2)$,
$(13, 10, 4, 1)$,
$(13, 14, 5)$,
$(13, 21, 1)$,
$(15, 9, 4, 1)$,
$(15, 12, 3, 1)$,
$(15, 13, 5)$,
$(15, 18, 3)$,\\
& $(15, 20, 1)$,
$(15, 22)$,
$(16, 7, 4, 1)$,
$(16, 10, 3, 1)$,
$(16, 11, 5)$,
$(16, 12, 2, 1)$,
$(16, 16, 3)$,
$(16, 19, 1)$,
$(16, 21)$,
$(17, 12, 4)$,
$(17, 14, 3)$,\\
&$(17, 16, 2)$,
$(17, 18, 1)$,
$(17, 20)$,
$(18, 13, 3)$,
$(18, 14, 2)$,
$(20, 8, 4)$,
$(20, 10, 3)$,
$(20, 13, 2)$,
$(20, 14, 1)$,
$(20, 18)$,
$(21, 10, 2)$,\\
&$(21, 15)$,
$(22, 13)$\\
$6$ & $( 8, 12, 8)$,
$( 10, 11, 4, 1)$,
$( 11, 12, 7)$,
$( 12, 10, 7)$,
$( 12, 13, 5)$,
$( 12, 18, 4)$,
$( 13, 16, 4)$,
$( 14, 9, 4, 1)$,
$( 14, 9, 7)$,
$( 14, 12, 6)$,
$( 14, 16, 3)$,\\
&$( 14, 19, 1)$,
$( 14, 21)$
$( 15, 7, 4, 1)$,
$( 15, 10, 3, 1)$,
$( 15, 10, 6)$,
$( 15, 11, 2, 1)$,
$( 15, 12, 5)$,
$( 15, 15, 4)$,
$( 15, 20)$,
$( 16, 7, 3, 1)$,
$( 16, 8, 6)$,\\
&$( 16, 9, 2, 1)$,
$( 16, 10, 5)$,
$( 16, 12, 1, 1)$,
$( 16, 13, 4)$,
$( 16, 14, 3)$,
$( 16, 18, 2)$,
$( 16, 19)$,
$( 17, 9, 5)$,
$( 17, 10, 4)$,
$( 17, 13, 3)$,
$( 17, 15, 2)$,
$( 17, 17, 1)$,\\
&$( 17, 18)$,
$( 18, 13, 2)$,
$( 18, 16, 1)$,
$( 18, 17)$,
$( 19, 9, 4)$,
$( 19, 12, 3)$,
$( 19, 15, 1)$,
$( 20, 7, 4)$,
$( 20, 9, 3)$,
$( 20, 12, 2)$,
$( 20, 13, 1)$,
$( 20, 15)$,
$( 21, 8, 3)$,\\
&$( 21, 9, 2)$,
$( 21, 12, 1)$,
$( 21, 14)$,
$( 22, 7, 3)$,
$( 22, 8, 2)$,
$( 22, 10, 1)$,
$( 23, 9, 1)$,
$( 24, 7, 2)$,
$( 24, 8, 1)$,
$( 24, 12)$,
$( 25, 9)$,
$( 26, 7)$\\
$7$ & $(8, 6, 8)$,
$(11, 9, 4, 1)$,
$(11, 12, 6)$,
$(12, 8, 4, 1)$,
$(12, 8, 6)$,
$(12, 12, 3, 1)$,
$(12, 12, 5)$,
$(12, 13, 4)$,
$(12, 15, 3)$,
$(12, 17, 2)$,
$(13, 7, 4, 1)$,\\
&$(13, 10, 3, 1)$,
$(13, 11, 5)$,
$(13, 12, 2, 1)$,
$(13, 12, 4)$,
$(13, 14, 3)$,
$(13, 16, 2)$,
$(14, 6, 4, 1)$,
$(14, 6, 7)$,
$(14, 9, 5)$,
$(14, 10, 2, 1)$,
$(14, 12, 1, 1)$,\\
&$(14, 17, 1)$,
$(14, 19)$,
$(15, 7, 5)$,
$(15, 8, 3, 1)$,
$(15, 9, 2, 1)$,
$(15, 11, 1, 1)$,
$(15, 11, 4)$,
$(15, 13, 3)$,
$(15, 16, 1)$,
$(16, 6, 3, 1)$,
$(16, 6, 6)$,\\
&$(16, 8, 2, 1)$,
$(16, 10, 1, 1)$,
$(16, 10, 4)$,
$(16, 12, 0, 1)$,
$(16, 12, 3)$
$(16, 15, 2)$,
$(16, 17)$,
$(17, 6, 5)$,
$(17, 7, 4)$,
$(17, 11, 3)$,
$(17, 13, 2)$,
$(17, 14, 1)$,\\
&$(17, 16)$,
$(18, 10, 3)$,
$(18, 13, 1)$,
$(18, 15)$,
$(19, 9, 3)$,
$(20, 6, 4)$,
$(20, 11, 2)$,
$(20, 12, 1)$,
$(20, 14)$,
$(21, 8, 2)$,
$(21, 10, 1)$,
$(21, 12)$,\\
&$(22, 9, 1)$,
$(22, 11)$,
$(23, 6, 3)$,
$(23, 7, 1)$,
$(23, 10)$,
$(24, 6, 2)$,
$(24, 9)$,
$(25, 6, 1)$,
$(25, 8)$,
$(26, 3, 1)$,
$(28, 6)$,
$(29, 3)$,
$(30, 1)$\\
$8$ &  $(8, 0, 8)$,
$(8, 9, 7)$,
$(8, 12, 6)$,
$(9, 9, 4, 1)$,
$(9, 10, 6)$,
$(9, 12, 3, 1)$,
$(9, 12, 5)$,
$(9, 13, 4)$,
$(9, 15, 3)$,
$(10, 7, 4, 1)$,
$(10, 10, 3, 1)$,
$(10, 10, 5)$,\\
&$(10, 12, 2, 1)$,
$(10, 12, 4)$,
$(10, 13, 3)$,
$(10, 15, 2)$,
$(11, 6, 4, 1)$,
$(11, 9, 6)$,
$(11, 10, 2, 1)$,
$(11, 11, 4)$,
$(12, 7, 6)$,
$(12, 9, 3, 1)$,
$(12, 9, 5)$,\\
&$(12, 10, 4)$,
$(12, 12, 1, 1)$,
$(12, 14, 2)$,
$(12, 16, 1)$,
$(12, 18)$,
$(13, 7, 3, 1)$,
$(13, 7, 5)$,
$(13, 9, 2, 1)$,
$(13, 12, 0, 1)$,
$(13, 12, 3)$,
$(14, 0, 7)$,\\
&$(14, 6, 6)$,
$(14, 7, 2, 1)$,
$(14, 8, 1, 1)$,
$(14, 9, 4)$
$(14, 11, 0, 1)$,
$(14, 11, 3)$,
$(14, 13, 2)$,
$(14, 15, 1)$,
$(14, 17)$,
$(15, 6, 3, 1)$,
$(15, 6, 5)$,\\
&$(15, 7, 1, 1)$,
$(16, 0, 6)$,
$(16, 4, 3, 1)$,
$(16, 4, 5)$,
$(16, 6, 2, 1)$,
$(16, 8, 4)$,
$(16, 9, 0, 1)$,
$(16, 10, 3)$,
$(16, 12, 2)$,
$(16, 14, 1)$,
$(16, 16)$
$(17, 0, 5)$,\\
&$(17, 3, 4)$,
$(17, 8, 3)$,
$(17, 10, 2)$,
$(17, 12, 1)$,
$(17, 14)$,
$(18, 9, 2)$,
$(18, 11, 1)$,
$(18, 12)$,
$(19, 6, 3)$,
$(19, 8, 2)$,
$(20, 0, 4)$,
$(20, 4, 3)$,
$(20, 7, 2)$,\\
&$(20, 9, 1)$,
$(20, 11)$,
$(21, 4, 2)$,
$(21, 7, 1)$,
$(22, 3, 2)$,
$(22, 6, 1)$,
$(22, 9)$,
$(23, 0, 3)$,
$(23, 4, 1)$,
$(24, 0, 2)$,
$(24, 3, 1)$,
$(24, 8)$,
$(25, 1, 1)$,\\
&$(25, 6)$,
$(26, 0, 1)$,
$(26, 4)$,
$(28, 3)$,
$(32)$\\
$9$ & $(8, 10, 4)$,
$(9, 9, 4)$,
$(9, 12, 3)$,
$(10, 8, 4)$,
$(10, 10, 3)$,
$(10, 12, 2)$,
$(10, 13, 1)$,
$(10, 15)$,
$(11, 11, 2)$,
$(12, 7, 4)$,
$(12, 9, 3)$,
$(12, 12, 1)$,
$(12, 14)$,\\
&$(13, 7, 3)$,
$(13, 10, 2)$,
$(14, 9, 2)$,
$(14, 11, 1)$,
$(14, 13)$,
$(15, 6, 3)$,
$(16, 0, 4)$,
$(16, 4, 3)$,
$(16, 8, 2)$,
$(16, 10, 1)$,
$(16, 12)$,
$(17, 3, 3)$,
$(17, 6, 2)$,\\
&$(17, 8, 1)$,
$(17, 10)$,
$(18, 2, 3)$
$(18, 4, 2)$,
$(18, 7, 1)$,
$(18, 9)$,
$(19, 0, 3)$,
$(19, 3, 2)$,
$(19, 6, 1)$,
$(20, 1, 2)$,
$(20, 5, 1)$,\\
&$(20, 8)$,
$(21, 4, 1)$,
$(21, 6)$,
$(22, 1, 1)$,
$(22, 5)$,
$(24, 4)$,
$(25, 2)$,
$(28)$\\
$10$ & $(8, 6, 4)$,
$(8, 8, 3)$,
$(9, 7, 3)$,
$(9, 10, 2)$,
$(9, 11, 1)$,
$(9, 13)$,
$(10, 5, 4)$,
$(10, 9, 2)$,
$(10, 12)$,
$(11, 6, 3)$,
$(12, 4, 4)$,
$(12, 5, 3)$,
$(12, 7, 2)$,
$(12, 10, 1)$,\\
&$(12, 11)$,
$(13, 6, 2)$,
$(13, 8, 1)$,
$(13, 10)$,
$(14, 3, 3)$,
$(14, 5, 2)$,
$(14, 9)$,
$(15, 2, 3)$,
$(15, 7, 1)$,
$(16, 4, 2)$,
$(16, 6, 1)$,
$(16, 8)$,
$(17, 4, 1)$,
$(17, 6)$,\\
&$(18, 2, 1)$,
$(18, 5)$,
$(20, 4)$,
$(21, 2)$,
$(22, 1)$,
$(24)$\\
$11$ & $(4, 6, 4)$,
$(6, 5, 4)$,
$(7, 6, 3)$,
$(8, 4, 4)$,
$(8, 5, 3)$,
$(9, 6, 2)$,
$(9, 8, 1)$,
$(9, 10)$,
$(10, 3, 3)$,
$(10, 5, 2)$,
$(10, 9)$,
$(11, 2, 3)$,
$(11, 7, 1)$,
$(12, 4, 2)$,\\
&$(12, 6, 1)$,
$(12, 8)$,
$(13, 4, 1)$,
$(13, 6)$,
$(14, 2, 1)$,
$(14, 5)$,
$(16, 4)$,
$(17, 2)$,
$(18, 1)$,
$(20)$\\
$12$ & $(4, 3, 3)$,
$(6, 2, 3)$,
$(6, 5, 2)$,
$(6, 7, 1)$,
$(6, 9)$,
$(8, 4, 2)$,
$(8, 6, 1)$,
$(8, 8)$,
$(9, 4, 1)$,
$(9, 6)$,
$(10, 2, 1)$,
$(10, 5)$,
$(12, 4)$\\
&$(13, 2)$,
$(14, 1)$,
$(16)$ \\
$13$ & $(6, 5)$,
$(8, 4)$,
$(9, 2)$,
$(10, 1)$,
$(12)$ \\
$14$ & $(4, 3)$,
$(5, 2)$,
$(6, 1)$,
$(8)$\\
$15$& $(4)$\\
$16$ & $(0)$\\
\hline
\end{tabular}
}
\caption{
The Pareto-optimal statistics $(a,b,c,d,e)$ of Moser sets in $[3]^4$.  To save space, all statistics with the same value of $a$ have been collected in a single row; also, trailing zeroes for $(b,c,d,e)$ have been dropped, thus for instance $(b,c)$ is short for $(b,c,0,0)$.  This table can also be found at {\tt http://spreadsheets.google.com/ccc?key=rwXB\_Rn3Q1Zf5yaeMQL-RDw}.}
\label{4d-pareto}
\end{figure}

\begin{proof}  This was computed by computer search as follows.  First, one observed that if $(a,b,c,d,e)$ was Pareto-optimal, then $a\geq 3$.  To see this, it suffices to show that for any Moser set $A \subset [3]^4$ with $a(A)=0$, it is possible to add three points from $S_{4,4}$ to $A$ and still have a Moser set.  To show this, suppose first that $A$ contains a point from $S_{1,4}$, such as $2221$. Then $A$ must omit either $2211$ or $2231$; without loss of generality we may assume that it omits $2211$. Similarly we may assume it omits $2121$ and $1221$. Then we can add $1131$, $1311$, $3111$ to $A$, as required. Thus we may assume that $A$ contains no points from $S_{1,4}$.  Now suppose that $A$ omits a point from $S_{2,4}$, such as $2211$. Then one can add $3333$, $3111$, $1311$ to $A$, as required. Thus we may assume that A contains all of $S_{2,4}$, which forces $A$ to omit $2222$, as well as at least one point from $S_{3,4}$, such as $2111$. But then $3111$, $1111$, $3333$ can be added to the set, a contradiction. 

Thus we only need to search through sets $A \subset [3]^4$ for which $|A \cap S_{4,4}| \geq 3$.  A straightforward computer search shows that up to the symmetries of the cube, there are $391$ possible choices for $A \cap S_{4,4}$.  For each such choice, we looped through all the possible values of the slices $A \cap 1***$ and $A \cap 3***$, i.e. all three-dimensional Moser sets which had the indicated intersection with $S_{3,3}$.  (For fixed $A \cap S_{4,4}$, the number of possibilities for $A \cap 1***$ ranges from $1$ to $87123$, and similarly for $A \cap 3***$).  For each pair of slices $A \cap 1***$ and $A \cap 3***$, we computed the lines connecting these two sets to see what subset of $2***$ was excluded from $A$; there are $2^{27}$ possible such exclusion sets.  We precomputed a lookup table that gave the Pareto-optimal statistics for $A \cap 2***$ for each such choice of exclusion set; using this lookup table for each choice of $A \cap 1***$ and $A \cap 3***$ and collating the results, we obtained the above list. On a Linux cluster, the lookup table took 22 minutes to create, and the loop over the $A \cap 1***$ and $A \cap 3***$ slices took two hours, spread out over $391$ machines (one for each choice of $A \cap S_{4,4}$). Further details (including source code) can be found at the page {\tt 4D Moser brute force search} of \cite{polywiki}.
\end{proof}

As a consequence of this data, we have the following facts about the statistics of large Moser sets:

\begin{proposition}\label{stat} Let $A \subset [3]^4$ be a Moser set with statistics $(a,b,c,d,e)$.
\begin{itemize}
\item[(i)] If $|A| \geq 40$, then $e=0$.
\item[(ii)] If $|A| \geq 43$, then $d=0$.
\item[(iii)] If $|A| \geq 42$, then $d \leq 2$.
\item[(iv)] If $|A| \geq 41$, then $d \leq 3$.
\item[(v)] If $|A| \geq 40$, then $d \leq 6$.
\item[(vi)] If $|A| \geq 43$, then $c \geq 18$.
\item[(vii)] If $|A| \geq 42$, then $c \geq 12$.
\item[(viii)] If $|A| \geq 43$, then $b \geq 15$.
\end{itemize}
\end{proposition}

\begin{remark}  This proposition was first established by an integer program, see the file {\tt integer.tex} at \cite{polywiki}.  A computer-free proof can be found at 

\centerline{{\tt terrytao.files.wordpress.com/2009/06/polymath2.pdf}.}
\end{remark}

\subsection{Five dimensions}

Now we establish the bound $c'_{5,3}=124$.  In view of the lower bounds, it suffices to show that there does not exist a Moser set $A \subset [3]^5$ with $|A|=125$.  

We argue by contradiction.  Let $A$ be as above, and let $(a(A),\ldots,f(A))$ be the statistics of $A$.

\begin{lemma}\label{fvan} $f(A)=0$.
\end{lemma}

\begin{proof} If $f(A)$ is non-zero, then $A$ contains $22222$, then each of the $\frac{3^5-1}{2} = 121$ antipodal pairs in $[3]^5$ can have at most one point in $A$, leading to only $122$ points.
\end{proof}

Let us slice $[3]^5$ into three parallel slices, e.g. $1****, 2****, 3****$.  The intersection of $A$ with each of these slices has size at most $43$.  In particular, this implies that
\begin{equation}\label{boo}
 |A \cap 1****| + |A \cap 3****| = 125 - |A \cap 2****| \geq 82.
\end{equation}
Thus at least one of $A \cap 1****$, $A \cap 3****$ has cardinality at least $41$; by Proposition \ref{stat}(iv) we conclude that
\begin{equation}\label{d13}
\min( d(1****), d(3****) ) \leq 3.
\end{equation}
Furthermore, equality can only hold in \eqref{d13} if $A \cap 1****$, $A \cap 3****$ both have cardinality exactly $41$, in which case from Proposition \ref{stat}(iv) again we must have
\begin{equation}\label{d13a}
d(1****)=d(3****)=3.
\end{equation}
Of course, we have a similar result for permutations.

Now we improve the bound $|A \cap 2****| \leq 43$:

\begin{lemma} $|A \cap 2****| \leq 41$.
\end{lemma}

\begin{proof} Suppose first that $|A \cap 2****|=43$.  Let $A' \subset [3]^4$ be the subset of $[3]^4$ corresponding to $A \cap 2****$, thus $A'$ is a Moser set of cardinality $43$.  By Proposition \ref{stat}(vi), $c(A') \geq 18$.  By Lemma \ref{dci}, the sum of the $c(V)$, where $V$ ranges over the eight side slices of $[3]^4$, is therefore at least $36$.  By the pigeonhole principle, we may thus find two opposing side slices, say $1***$ and $3***$, with $c(1***)+c(3****) \geq 9$.  Since $c(1***), c(3***)$ cannot exceed $6$, we thus have $c(1***), c(3***) \geq 3$, with at least one of $c(1***), c(3***)$ being at least $5$.  Passing back to $A$, this implies that $d(*1***), d(*3***) \geq 3$, with at least one of $d(*1***), d(*3***)$ being at least $5$.  But this contradicts \eqref{d13} together with the refinement \eqref{d13a}.

We have just shown that $|A \cap 2****| \leq 42$; we can thus improve \eqref{boo} to
$$ |A \cap 1****| + |A \cap 3****| \geq 83.$$
Combining this with Proposition \ref{stat}(ii)-(v) we see that
\begin{equation}\label{d13-6}
 d(1****)+d(3****) \leq 6
\end{equation}
with equality only if $|A \cap 2****|=42$, and similarly for permutations.

Now let $A'$ be defined as before.  Then we have
$$ c(1***) + c(3***) \leq 6$$
and similarly for permutations.  Applying Lemma \ref{dci}, this implies that $c(2****) = c(A') \leq 12$.

Now suppose for contradiction that $|A'|=|A \cap 2****|=42$.  Then by Proposition \ref{stat}(vii) we have 
\begin{equation}\label{coo-1}
c(2****) = 12; 
\end{equation}
applying Lemma \ref{dci} again, this forces $c(1***)+c(3***)=6$ and similarly for permutations, which then implies that
\begin{equation}\label{doo}
d(*1***)+d(*3***) = d(**1**)+d(**3**) = d(***1*)+d(***3*) = d(****1)+d(****3) = 6
\end{equation}
and hence
$$ |A \cap *2***| = |A \cap **2**| = |A \cap ***2*| = |A \cap ****2| = 42$$
and thus
\begin{equation}\label{coo-2}
c(*2***) = c(**2**) = c(***2*) = c(****2) = 12.
\end{equation}
Combining \eqref{coo-1}, \eqref{doo}, \eqref{coo-2} we conclude that
$$ d(1****)+d(3****) = 16,$$
contradicting \eqref{d13-6}.
\end{proof}

With this proposition, the bound \eqref{boo} now improves to
\begin{equation}\label{84}
|A \cap 1****| + |A \cap 3****| \geq 84
\end{equation}
and in particular
\begin{equation}\label{41}
|A \cap 1****|, |A \cap 3****| \geq 41.
\end{equation}
from this and Proposition \ref{stat}(ii)-(iv) we now have
\begin{equation}\label{d13-improv}
 d(1****)+d(3****) \leq 4
\end{equation}
and similarly for permutations.

\begin{lemma}\label{evan} $e(A)=0$.
\end{lemma}

\begin{proof} From \eqref{84}, the intersection of $A$ with any side slice has cardinality at least $41$, and thus by Proposition \ref{stat}(i) such a side slice has an $e$-statistic of zero.  The claim then follows from Lemma \ref{dci}.
\end{proof}

We need a technical lemma:

\begin{lemma}\label{tech} Let $B \subset S_{5,5}$.  Then there exist at least $|B|-4$ pairs of strings in $B$ which differ in exactly two positions.
\end{lemma}

\begin{proof} The first non-vacuous case is $|B|=5$.  It suffices to establish this case, as the higher cases then follow by induction (locating a pair of the desired form, then deleting one element of that pair from $B$).

Suppose for contradiction that one can find a $5$-element set $B \subset S_{5,5}$ such that no two strings in $B$ differ in exactly two positions.  Recall that we may split $S_{5,5}=S_{5,5}^e \cup S_{5,5}^o$, where $S_{5,5}^e$ are those strings with an even number of $1$'s, and $S_{5,5}^o$ are those strings with an odd number of $1$'s.  By the pigeonhole principle and symmetry we may assume $B$ has at least three elements in $S_{5,5}^o$.  Without loss of generality, we can take one of them to be $11111$, thus excluding all elements in $S_{5,5}^o$ with exactly two $3$s, leaving only the elements with exactly four $3$s.  But any two of them differ in exactly two positions, a contradiction.
\end{proof}

We can now improve the trivial bound $c(A) \leq 80$:

\begin{corollary}[Non-maximal $c$]\label{cmax} $c(A) \leq 79$.  If $a(A) \geq 7$, then $c(A) \leq 78$.
\end{corollary}

\begin{proof} If $c(A)=80$, then $A$ contains all of $S_{3,5}$, which then implies that no two elements in $A \cap S_{5,5}$ can differ in exactly two places.  It also implies (from \eqref{alpha-1}) that $d(A)$ must vanish, and that $b(A)$ is at most $40$. By Lemma \ref{tech}, we also have that $a(A) = |A \cap S_{5,5}|$ is at most $4$.  Thus $|A| \leq 4 + 40 + 80 + 0 + 0 = 124$, a contradiction.

Now suppose that $a(A) \geq 7$.  Then by Lemma \ref{tech} there are at least three pairs in $A \cap S_{5,5}$ that differ in exactly two places.  Each such pair eliminates one point from $A \cap S_{3,5}$; but each point in $S_{3,5}$ can be eliminated by at most two such pairs, and so we have at least two points eliminated from $A \cap S_{3,5}$, i.e. $c(A) \leq 78$ as required.
\end{proof}

Next, we rewrite the quantity $125=|A|$ in terms of side slices.  From Lemmas \ref{fvan}, \ref{evan} we have
$$ a(A) + b(A) + c(A) + d(A) = 125$$
and hence by Lemma \ref{dci}, the quantity
$$ s(V) := a(V) + \frac{5}{4} b(V) + \frac{5}{3} c(V) + \frac{5}{2} d(V) - \frac{125}{2},$$
where $V$ ranges over side slices, has an average value of zero.  

\begin{proposition}[Large values of $s(V)$]\label{suv}  For all side slices, we have $s(V) \leq 1/2$.  Furthermore, we have $s(V) < -1/2$ unless the statistics $(a(V), b(V), c(V), d(V), e(V))$ are of one of the following four cases:
\begin{itemize}
\item (Type 1) $(a(V),b(V),c(V),d(V),e(V)) = (2,16,24,0,0)$ (and $s(V) = -1/2$ and $|A \cap V| = 42$);
\item (Type 2) $(a(V),b(V),c(V),d(V),e(V)) = (4,16,23,0,0)$ (and $s(V) = -1/6$ and $|A \cap V| = 43$);
\item (Type 3) $(a(V),b(V),c(V),d(V),e(V)) = (4,15,24,0,0)$ (and $s(V) = 1/4$ and $|A \cap V| = 43$);
\item (Type 4) $(a(V),b(V),c(V),d(V),e(V)) = (3,16,24,0,0)$ (and $s(V) = 1/2$ and $|A \cap V| = 43$);
\end{itemize}
\end{proposition}

\begin{proof}
Let $V$ be a side slice.  From \eqref{41} we have
$$ 41 \leq a(V)+b(V)+c(V)+d(V) = |A \cap V| \leq 43.$$
First suppose that $|A \cap V| = 43$, then from Proposition \ref{stat}(ii), (viii), $d(V)=0$ and $b(V) \geq 15$.
Also, we have the trivial bound $c(V) \leq 24$, together with the inequality
$$ 3b(V) + 2c(V) \leq 96$$
from \eqref{alpha-1}.  To exploit these facts, we rewrite $s(V)$ as
$$ s(V) = \frac{1}{2} - \frac{1}{2}( 24 - c(V) ) - \frac{1}{12} (96-3b(V)-2c(V)).$$
Thus $s(V) \leq 1/2$ in this case.  If $s(V) \geq -1/2$, then
$$ 6 (24-c(V)) + (96-3b(V)-2c(V)) \leq 12,$$
which together with the inequalities $b(V) \leq 15$, $c(V) \leq 24$, $3b(V)+2c(V) \leq 96$ we conclude that $(b(V),c(V))$ must be one of $(16,24)$, $(15, 24)$, $(16, 23)$, $(15, 23)$.  The first three possibilities lead to Types 4,3,2 respectively.  The fourth type would lead to $(a(V),b(V),c(V),d(V),e(V)) = (5,15,23,0,0)$, but this contradicts \eqref{eleven}.

Next, suppose $|A \cap V| = 42$, so by Proposition \ref{stat}(iii) we have $d(V) \leq 2$.  From \eqref{alpha-1} we have
\begin{equation}\label{2cd}
2c(V) + 3d(V) \leq 48
\end{equation}
while from \eqref{alpha-2} we have
\begin{equation}\label{3cd}
3b(V)+2c(V)+3d(V) \leq 96
\end{equation}
and so we can rewrite $s(V)$ as
\begin{equation}\label{sv2}
s(V) = -\frac{1}{2} - \frac{1}{4}( 48 - 2c(V) - 3d(V) ) - \frac{1}{12} (96-3b(V)-2c(V)-3d(V)) + \frac{1}{2} d(V).
\end{equation}
This already gives $s(V) \leq 1/2$.  If $d(V)=0$, then $s(V) \leq -1/2$, with equality only in Type 1.  If $d(V)=1$, then the set $A' \subset [3]^4$ corresponding to $A \cap V$ contains a point in $S_{3,4}$, which without loss of generality we can take to be $2221$.  Considering the three lines $*221$, $2*21$, $22*1$, we see that at least three points in $S_{2,4}$ must be missing from $A'$, thus $c(V) \leq 21$.  This forces $48-2c(V)-3d(V) \geq 3$, and so $s(V) < -3/4$.  Finally, if $d(V)=2$, then $A'$ contains two points in $S_{3,4}$.  If they are antipodal (e.g. $2221$ and $2223$), the same argument as above shows that at least six points in $S_{2,4}$ are missing from $A'$; if they are not antipodal (e.g. $2221$ and $2212$) then by considering the lines $*221$, $2*21$, $22*1$, $*212$, $2*12$ we see that five points are missing.  Thus we have $c(V) \leq 19$, which forces $48-2c(V)-3d(V) \geq 4$.  This forces $s(V) \leq -1/2$, with equality only when $c(V)=19$ and $3b(V)+2c(V)+3d(V)=96$, but this forces $b(V)$ to be the non-integer $52/3$, a contradiction, which concludes the treatment of the $|A \cap V|=42$ case.

Finally, suppose $|A \cap V| = 41$.  Using \eqref{2cd}, \eqref{3cd} as before we have
\begin{equation}\label{sv3}
 s(V) = -\frac{3}{2} - \frac{1}{4}( 48 - 2c(V) - 3d(V) ) - \frac{1}{12} (96-3b(V)-2c(V)-3d(V)) + \frac{1}{2} d(V),
\end{equation}
while from Proposition \ref{stat}(vi) we have $d(V) \leq 3$.  This already gives $s(V) \leq 0$, and $s(V) \leq -1$ when $d(V)=1$.  In order to have $s(V) \geq -1/2$, we must then have $d(V)=2$ or $d(V)=3$.  But then the arguments of the preceding paragraph give $48-2c(V)-3d(V) \geq 4$, and so $s(V) \leq -1$ in this case.
\end{proof}

Since the $s(V)$ average to zero, by the pigeonhole principle we may find two opposing side slices (e.g. $1****$ and $3****$), whose total $s$-value is non-negative.  Actually we can do a little better:

\begin{lemma}\label{side-off} There exists two opposing side slices whose total $s$-value is strictly positive.
\end{lemma}

\begin{proof} If this is not the case, then we must have $s(1****)+s(3****)=0$ and similarly for permutations.  Using Proposition \ref{suv} we thus see that for every opposing pair of side slices, one is Type 1 and one is Type 4.  In particular $c(V)=24$ for all side slices $V$.  But then by Lemma \ref{dci} we have $c(A)=80$, contradicting Lemma \ref{cmax}.
\end{proof}

Let $V, V'$ be the side slices in Lemma \ref{side-off}
By Proposition \ref{suv}, the $V, V'$ slices must then be either Type 2, Type 3, or Type 4, and they cannot both be Type 2.  Since $a(A) = a(V)+a(V')$, we conclude
\begin{equation}\label{amix}
6 \leq a(A) \leq 8.
\end{equation}
In a similar spirit, we have
$$ c(V) + c(V') \leq 23+24.$$
On the other hand, by considering the $24$ lines connecting $c$-points of $V, V'$ to $c$-points of the centre slice $W$ between $V$ and $V'$, each of which contains at most two points in $A$, we have
$$ c(V) + c(W) + c(V') \leq 24 \times 2.$$
Thus $c(W) \leq 1$; since
$$ d(A) = d(V) + d(V') + c(W)$$
we conclude from Proposition \ref{suv} that $d(A) \leq 1$.  Actually we can do better:

\begin{lemma} $d(A)=0$.
\end{lemma}

\begin{proof} Suppose for contradiction that $d(A)=1$; without loss of generality we may take $11222 \in A$.  This implies that $d(1****)=d(*1***)=1$.  Also, by the above discussion, $c(**1**)$ and $c(**3**)$ cannot both be $24$, so by Proposition \ref{suv}, $s(**1**)+s(**3**) \leq 1/3$; similarly
$s(***1*)+s(***3*) \leq 1/3$ and $s(****1)+s(****3) \leq 1/3$.  Since the $s$ average to zero, we see from the pigeonhole principle that either $s(1****)+s(3****) \geq -1/2$ or $s(*1***)+s(*3***) \geq -1/2$.  We may assume by symmetry that 
\begin{equation}\label{star-2}
s(1****)+s(3****) \geq -1/2.
\end{equation}
Since $s(3****) \leq 1/2$ by Proposition \ref{suv}, we conclude that
\begin{equation}\label{star}
 s(1****) \geq -1.
\end{equation}
If $|A \cap 1****|=41$, then by \eqref{sv3} we have
$$ s(1****) = -1 - \frac{1}{4}( 48 - 2c(1****) - 3d(1****) ) - \frac{1}{12} (96-3b(1****)-2c(1****)-3d(1****))$$
but the arguments in Proposition \ref{suv} give $48 - 2c(1****) - 3d(1****) \geq 3$ and $96-3b(1****)-2c(1****)-3d(1****) \geq 0$, a contradiction.  So we must have $|A \cap 1****|=42$ (by Proposition \ref{stat}(ii) and \eqref{41}).  In that case, from \eqref{sv2} we have
$$ s(1****) = \frac{1}{4}( 48 - 2c(1****) - 3d(1****) ) - \frac{1}{12} (96-3b(1****)-2c(1****)-3d(1****))$$
while also having $48 - 2c(1****) - 3d(1****) \geq 3$ and $96-3b(1****)-2c(1****)-3d(1****) \geq 0$.  Since $s(1****) \geq -1$ and $d(1****)=1$, we soon see that we must have $48 - 2c(1****) - 3d(1****) = 3$ and $96-3b(1****)-2c(1****)-3d(1****) \leq 3$, which forces $c(1****)=21$ and $b(1****)=16$ or $b(1****)=17$; thus the statistics of $1****$ are either $(4,16,21,1,0)$ or $(3,17,21,1,0)$.

We first eliminate the $(3,17,21,1,0)$ case.  In this case $s(1****)$ is exactly $-1$.  Inspecting the proof of \eqref{star}, we conclude that $s(3****)$ must be $+1/2$ and that $s(**1**)+s(**3**)=1/3$.  From the former fact and Proposition \ref{suv} we see that $a(A) = a(1****)+a(3****)=3+3=6$; on the other hand, from the latter fact and Proposition \ref{suv} we have $a(A) = a(**1**)+a(**3**) = 4+3=7$, a contradiction.

So $1****$ has statistics $(4,16,21,1,0)$, which implies that $s(1****)=-3/4$ and $|A \cap 1****|=42$.  By \eqref{star-2} we conclude 
\begin{equation}\label{s3}
s(3****) \geq 1/4,
\end{equation} 
which by Proposition \ref{suv} implies that $|A \cap 3****|=43$, and hence $|A \cap 2****|=40$.  On the other hand, since $e(A)=f(A)=0$ and $d(A)=1$, with the latter being caused by $11222$, we see that $c(2****)=d(2****)=e(2****)=0$.  From \eqref{alpha-1} we have $4a(2****)+b(2****) \leq 64$, and we also have the trivial inequality $b(2****) \leq 32$; these inequalities are only compatible if $2****$ has statistics $(8,32,0,0,0)$, thus $A \cap 2****$ contains $S_{2,5} \cap 2****$.

If $a(3****)=4$, then $a(A)=a(1****)+a(3****)=8$, which by Proposition \ref{suv} implies that $s(**1**)+s(**3**)$ cannot exceed $1/12$, and similarly for permutations.  On the other hand, from Proposition \ref{suv} $s(**1**)+s(**3**)$ cannot exceed $-3/4 + 1/4 = -1/2$, and so the average value of $s$ cannot be zero, a contradiction.  Thus $a(3****) \neq 4$, which by \eqref{s3} and Proposition \ref{suv} implies that $**3**$ has statistics $(3,16,24,0,0)$.

In particular, $A$ contains $16$ points from $3**** \cap S_{1,5}$ and all of $3**** \cap S_{2,5}$.  As a consequence, no pair of the $16$ points in $A \cap 3**** \cap S_{1,5}$ can differ in only one coordinate; partitioning the $32$-point set $3**** \cap S_{1,5}$ into $16$ such pairs, we conclude that every such pair contains exactly one element of $A$.  We conclude that $A \cap 3**** \cap S_{1,5}$ is equal to either $3**** \cap S_{1,5}^e$ or $3**** \cap S_{1,5}^o$.

On the other hand, $A$ contains all of $2**** \cap S_{2,5}$, and exactly sixteen points from $1**** \cap S_{1,5}$.  Considering the vertical lines $*xyzw$ where $xyzw \in S_{1,4}$, we conclude that $A \cap 1**** \cap S_{1,5}$ is either equal to $1**** \cap S_{1,5}^o$ or $1**** \cap S_{1,5}^e$.
But either case is incompatible with the fact that $A$ contains $11222$ (consider either the line $11xx2$ or $11x\overline{x}2$, where $x=1,2,3$ and $\overline{x}=4-x$), obtaining the required contradiction.
\end{proof}

We can now eliminate all but three cases for the statistics of $A$:

\begin{proposition}[Statistics of $A$]  The statistics $(a(A),b(A),c(A),d(A),e(A),f(A))$ of $A$ must be one of the following three tuples:
\begin{itemize}
\item (Case 1) $(6,40,79,0,0)$;
\item (Case 2) $(7,40,78,0,0)$;
\item (Case 3) $(8,39,78,0,0)$.
\end{itemize}
\end{proposition}

\begin{proof}
Since $d(A)=e(A)=f(A)=0$, we have
$$ c(2****)=d(2****)=e(2****)=0.$$
On the other hand, from \eqref{alpha-1} we have $4a(2****)+b(2****) \leq 64$ as well as the trivial inequality $b(2****) \leq 24$, and also we have
$$ |A \cap 2****| = 125 - |A \cap 1****| - |A \cap 3****| \geq 125 - 43 - 43 = 39.$$
Putting all this together, we see that the only possible statistics for $2****$ are $(8,32,0,0,0)$, $(7,32,0,0,0)$, or $(8,31,0,0,0)$.  In particular, $7 \leq a(2****) \leq 8$ and $31 \leq b(2****) \leq 32$, and similarly for permutations. Applying Lemma \ref{dci} we conclude that
$$ 35 \leq b(A) \leq 40$$
and
$$ 77.5 \leq c(A) \leq 80.$$
Combining this with the first part of Corollary \ref{cmax} we conclude that $c(A)$ is either $78$ or $79$.  From this and \eqref{amix} we see that the only cases that remain to be eliminated are $(7,39,79,0,0)$ and $(8,38,79,0,0)$, but these cases are incompatible with the second part of Corollary \ref{cmax}.
\end{proof}

We now eliminate each of the three remaining cases in turn.

\subsection{Elimination of $(6,40,79,0,0)$}

Here $A \cap S_{5,5}$ has six points.  By Lemma \ref{tech}, there are at least two pairs in this set which differ in two positions.  Their midpoints are eliminated from $A \cap S_{3,5}$.  But $A$ omits exactly one point from $S_{3,5}$, so these midpoints must be the same.  By symmetry, we may then assume that these two pairs are $(11111,11133)$ and $(11113,11131)$.  Thus the eliminated point in $S_{3,5}$ is $11122$, i.e. $A$ contains $S_{3,5} \backslash \{11122\}$.  Also, $A$ contains $\{11111,11133,11113,11131\}$ and thus must omit $\{11121, 11123, 11112, 11132\}$.

Since $11322 \in A$, at most one of  $11312, 11332$ lie in $A$. By symmetry we may assume $11312 \not \in A$, thus there is a pair $(xy1z2, xy3z2)$ with $x,y,z = 1,3$ that is totally omitted from $A$, namely $(11112,11312)$. On the other hand, every other pair of this form can have at most one point in the $A$, thus there are at most seven points in $A$ of the form $xyzw2$ with $x,y,z,w = 1,3$. Similarly there are at most 8 points of the form $xyz2w$, or of $xy2zw$, $x2yzw$, $2xyzw$, leading to $b(A) \leq 7+8+8+8+8=39$, contradicting the statistic $b(A)=40$.

\subsection{Elimination of $(7,40,78,0,0)$}

Here $A \cap S_{5,5}$ has seven points.  By Lemma \ref{tech}, there are at least three pairs in this set which differ in two positions.  As we can only eliminate two points from $S_{3,5}$, two of the midpoints of these pairs must be the same; thus, as in the previous section, we may assume that $A$ contains $\{11111,11133,11113,11131\}$ and omits $\{11121, 11123, 11112, 11132\}$ and $11122$.

Now consider the $160$ lines $\ell$ connecting two points in $S_{4,5}$ to one point in $S_{3,5}$ (i.e. $*2xyz$ and permutations, where $x,y,z=1,3$).  By double counting, the total sum of $|\ell \cap A|$ over all $160$ lines is $4b(A)+2c(A) = 316 = 158 \times 2$.  On the other hand, each of these lines contain at most two points in $A$, but two of them (namely $1112*$ and $1112*$) contain no points.  Thus we must have $|\ell \cap A|=2$ for the remaining $158$ lines $\ell$.

Since $A$ omits $1112x$ and $111x2$ for $x=1,3$, we thus conclude (by considering the lines $11*2x$ and $11*x2$) that $A$ must contain $1132x$, $113x2$, $1312x$, and $131x2$.  Taking midpoints, we conclude that $A$ omits $11322$ and $13122$.  But together with $11122$ this implies that at least three points are missing from $A \cap S_{3,5}$, contradicting the hypothesis $c(A)=78$.

\subsection{Elimination of $(8,39,78,0,0)$}

Now $A \cap S_{5,5}$ has eight points.  By Lemma \ref{tech}, there are at least three pairs in this set which differ in two positions.  As we can only eliminate two points from $S_{3,5}$, two of these pairs $(a,b), (c,d)$ must have the same midpoint $p$, and two other pairs $(a',b'), (c',d')$ must have the same midpoint $p'$, and $A$ contains $S_{3,5} \backslash \{p,p'\}$.  As $p,p'$ are distinct, the plane containing $a,b,c,d$ is distinct from the plane containing $a',b',c',d'$.

Again consider the $160$ lines $\ell$ from the previous section.  This time, the sum of the $|\ell \cap A|$ is $4b(A)+2c(A) = 312 = 156 \times 2$.
But the two lines in the plane of $a,b,c,d$ passing through $p$, and the two lines in the plane of $a',b',c',d'$ passing through $p'$, have no points; thus we must have $|\ell \cap A|=2$ for the remaining $156$ lines $\ell$.

Without loss of generality we have $(a,b)=(11111,11133)$, $(c,d) = (11113,11131)$, thus $p = 11122$. By permuting the first three indices, we may assume that $p'$ is not of the form $x2y2z, x2yz2, xy22z, xy2z2$ for any $x,y,z=1,3$. Then we have $1112x \not \in A$ and $1122x \in A$ for every $x=1,3$, so by the preceding paragraph we have $1132x \in A$; similarly for $113x2, 1312x, 131x2$. Taking midpoints, this implies that $13122, 11322 \not \in A$, but this (together with 11122) shows that at least three points are missing from $A \cap S_{3,5}$, contradicting the hypothesis $c(A)=78$.

\subsection{Six dimensions}

Now we establish the bound $c'_{6,3}=353$.  In view of the lower bounds, it suffices to show that there does not exist a Moser set $A \subset [3]^5$ with $|A|=354$.

We argue by contradiction.  Let $A$ be as above, and let $(a(A),\ldots,g(A))$ be the statistics of $A$.

\begin{lemma}\label{g6} $g(A)=0$.
\end{lemma}
  
\begin{proof} For any four-dimensional slice $V$ of $A$, define 
$$S(V) := 15 a(V) + 5 b(V) + 5 c(V)/2 + 3d(V)/2 + e(V).$$
From Lemma \ref{dci} we see that $|A|$ is equal to $a(A)+b(A)$ plus the average of $S(V)$ where $V$ ranges over the twenty slices which are some permutation of the center slice $22****$.

If $g(A)=1$, then $a(A) \leq 32$ and $b(A) \leq 96$ by \eqref{alpha-1}.  Meanwhile, $e(V)=g(A)=1$ for every center slice $V$, so from Lemma \ref{paretop-4}, one can show that $S(V) \leq 223.5$ for every such slice.  We conclude that $|A| \leq 351.5$, a contradiction.
\end{proof}

For any four-dimensional slice $V$ of $A$, define the \emph{defects} to be
$$ D(V) := 356 - [4a(V)+6b(V)+10c(V)+20d(V)+60e(V)].$$
Define a \emph{corner slice} to be one of the permutations or reflections of $11****$, thus there are $60$ corner slices.  From Lemma \ref{dci} we see that $356-|A|+f(A)=2+f(A)$ is the average of the defects of all the $60$ corner slices.  On the other hand, from Lemma \ref{paretop-4} and a straightforward computation, one concludes

\begin{lemma}\label{defects}  Let $A$ be a four-dimensional Moser set.  Then $D(A) \geq 0$.  Furthermore:
\begin{itemize}
\item If $A$ has statistics $(6,12,18,4,0)$, then $D(A)=0$.
\item If $A$ has statistics $(5,12,18,4,0)$, $(5,12,12,4,1)$, or $(6,8,12,8,0)$, then $D(A)=4$.
\item For all other $A$, $D(A) \geq 6$.
\item If $a(A) = 4$, then $D(A) \geq 8$.
\item If $a(A) \geq 7$, then $D(A) \geq 16$ (with equality iff $A$ has statistics $(7,11,12,6,0)$).
\item If $a(A) \geq 8$, then $D(A) \geq 30$.
\item If $a(A) \geq 9$, then $D(A) \geq 86$.
\end{itemize}
\end{lemma}

Define a \emph{family} to be a set of four parallel corner slices, thus there are $15$ families, which are all a permutation of $\{11****, 13****, 31****, 33**** \}$.  We refer to the family $\{11****, 13****, 31****, 33**** \}$ as $ab****$, and similarly define the family $a*b***$, etc.

\begin{lemma}\label{f6} $f(A)=0$.
\end{lemma}

\begin{proof}  For any four-dimensional slice $V$ of $A$, define
$$ s(V) := 12 a(V)+15 b(V)/2+20 c(V)/3+15 d(V)/2 + 12 e(V),$$
and define an \emph{edge slice} to be one of the $60$ permutations or reflections of $12****$.  From double counting we see that $|A|-a(A)$ is equal to the average of the $60$ values of $s(V)$ as $V$ ranges over edge slices.

From Lemma \ref{paretop-4} one can verify that $s(V) \leq 336$, and that $s(V) \leq 296 = 336-40$ if $e(V)=1$.  The number of edge slices $V$ for which $e(V)=1$ is equal to $5f(A)$, and so the average value of the $s(V)$ is at most $336 - \frac{40 \times 5}{60} f(A)$, and so
$$ |A| - a(A) \leq 336 - \frac{40 \times 5}{60} f(A)$$
which we can rearrange (using $|A|=354$) as
$$ a(A) \geq 18 + \frac{10}{3} f(A).$$

Suppose first that $f(A) \geq 3$; then $a(A) \geq 28$.   Then in any given family, there is a corner slice with an $a$ value at least $9$, or four slices with $a$ value at least $7$, or two slices with $a$ value at least $8$, or one slice with $a$ value $8$ and two of $a$ value at least $7$, leading to a total defect of at least $60$ by Lemma \ref{defects}.  Thus the average defect is at least $15$; on the other hand, the average defect is $2+f(A) \leq 2+12$, a contradiction.

Now suppose that $f(A)=2$; then $a(A) \geq 25$.  This means that in any given family, one of the four corner slices has an $a$ value of at least $7$, and thus by Lemma \ref{defects} has a defect of at least $16$.  Thus the average defect is at least $4$; on the other hand, the average defect is $2+f(A)=4$.  From Lemma \ref{defects}, this implies that in any given family, three of the corner slices have statistics $(6,12,18,4,0)$ and the last one has statistics $(7,11,12,6,0)$.  But this forces $b(A)=70.5$ by double counting, which is absurd.

The remaining case is when $f(A)=1$.  Here we need a different argument.  Without loss of generality we may take $122222 \in A$.  The average defect among all $60$ slices is $2+f(A)=3$.  Equivalently, the average defect among all $15$ families is $12$.

First suppose that $a(A) \neq 24$.  Then in every family, at least one of the corner slices needs to have an $a$ value distinct from six, and so the average defect in each family is at least $4$.  Thus the five families $ab****, a*b***, a**b**, a***b*, a****b$ have an average defect of at most $28$, which implies that the ten corner slices beginning with $1$ (or equivalently, adjacent to an edge slice containing $122222$) is at most $14$.  In other words, if $(a,b,c,d,e)$ is the average of the statistics of these ten corner slices, then
$$ 4a+6b+10c+20d+60e \geq 342.$$
On the other hand, $(a,b,c,d,e)$ must lie in the convex hull of the statistics of four-dimensional Moser sets, which are described by Lemma \ref{paretop}.  Also, as $A$ contains $122222$, one has $c/24, d/8, e \leq 1/2$ by considering lines with centre $122222$.  Finally, from \eqref{six} and double-counting one has $7c/24 + 3d/8 + 3e + 1 \leq 6$.  Inserting these facts into a standard linear program yields a contradiction; indeed, the maximal value of $4a+6b+10c+20d+60e$ with these constraints is $338 \frac{2}{3}$, attained when $(a,b,c,d,e) = (\frac{17}{3}, 16, 12, 4, \frac{1}{3})$.

Finally, we consider the case when $f(A)=1$ and $a(A)=24$.  The preceding arguments allow the average defect of the ten corner slices beginning with $1$ to be as large as $18$, which implies that $4a+6b+10c+20d+60e \geq 338$.  Linear programming shows that this is not possible if $a \geq 6$, thus $a<6$.  But this forces one of the corner slices beginning with $3$ to have an $a$ value of at least $7$, and thus to have a defect of at least $16$ by Lemma \ref{paretop}.  Repeating the preceding arguments, this increases the lower bound for $4a+6b+10c+20d+60e$ by $\frac{16}{10}$, to $339.6$; but this is now inconsistent with the upper bound of $338 \frac{2}{3}$ from linear programming.
\end{proof}

As a consequence of the above lemma, we see that the average defect of all corner slices is $2$, or equivalently that the total defect of these slices is $120$.

Call a corner slice \emph{good} if it has statistics $(6,12,18,4,0)$, and \emph{bad} otherwise.  Thus good slices have zero defect, and bad slices have defect at least four.  Since the average defect of the $60$ corner slices is $2$, there are at least $30$ good slices.

One can describe the structure of the good slices completely:

\begin{lemma}\label{sixt}  The subset of $[3]^4$ consisting of the strings $1111$, $1113$, $3333$, $1332$, $1322$, $1222$, $3322$ and permutations is a Moser set with statistics $(6,12,18,4,0)$.  Conversely, every Moser set with statistics $(6,12,18,4,0)$ is of this form up to the symmetries of the cube $[3]^4$.
\end{lemma}

\begin{proof}  This can be verified by computer.  By symmetry, one assumes $1222$,$2122$,$2212$ and $2221$ are in the set.  Then $18$ of the $24$ `$c$' points with two $2$s must be included; it is quick to check that $1122$ and permutations must be the six excluded.  Next, one checks that the only possible set of six `$a$' points with no $2$s is $1111$, $1113$, $3333$ and permutations.  Lastly, in a rather longer computation, one finds there is only possible set of twelve `$b$' points, that is points with one $2$.  A computer-free proof can be found at the page {\tt Classification of $(6,12,18,4,0)$ sets} at \cite{polywiki}.
\end{proof}

As a consequence of this lemma, given any $x,y,z,w \in \{1,3\}$, there is a unique good Moser set in $[3]^4$ set whose intersection with $S_{1,4}$ is $\{x222, 2y22, 22z2, 222w\}$, and these are the only $16$ possibilities. Call this set the \emph{good set of type $xyzw$}. It consists of
\begin{itemize} 
\item The four points $x222, 2y22, 22z2, 222w$ in $S_{1,4}$;
\item All $24$ elements of $S_{2,4}$ except for $xy22, x2z2, x22w, 2yz2, 2y2w, 22zw$;
\item The twelve points $xYZ2$, $xY2W$, $x2ZW$, $XyZ2$, $Xy2W$, $2yZW$, $XYz2$, $X2zW$, $2YzW$, $XY2w$, $X2Zw$, $2YZw$ in $S_{3,4}$, where $X=4-x$, $Y=4-y$, $Z=4-z$, $W=4-w$;
\item The six points $xyzw, xyzW, xyZw, xYzw, Xyzw, XYZW$ in $S_{4,4}$.
\end{itemize} 

We can use this to constrain the types of two intersecting good slices:

\begin{lemma}\label{pqs} Suppose that the $pq****$ slice is of type $xyzw$, and the $p*r***$ slice is of type $x'y'z'w'$, where $p,q,r,x,y,z,w,x',y',z',w'$ are in $\{1,3\}$. Then $x'=x$ iff $q=r$, and $y'z'w'$ is equal to either $yzw$ or $YZW$. If $x=r$ (or equivalently if $x'=q$), then $y'z'w'=yzw$.
\end{lemma}

\begin{proof} By reflection symmetry we can take $p=q=r=1$. Observe that the $11****$ slice contains $111222$ iff $x=1$, and the $1*1***$ slice similarly contains $111222$ iff $x'=1$. This shows that $x=x'$.

Suppose now that $x=x'=1$. Then the $111***$ slice contains the three elements $111y22, 1112z2, 11122w$, and excludes $111Y22, 1112Z2, 11122W$, and similarly with the primes, which forces $yzw=y'z'w'$ as claimed.

Now suppose that $x=x'=3$. Then the $111***$ slice contains the two elements $111yzw, 111YZW$, but does not contain any of the other six points in $S_{6,6} \cap 111***$, and similarly for the primes. Thus $y'z'w'$ is equal to either $yzw$ or $YZW$ as claimed. 
\end{proof}

Now we look at two adjacent parallel good slices, such as $11****$ and $13****$.  The following lemma asserts that such slices either have opposite type, or else will create a huge amount of defect in other slices:

\begin{lemma}\label{l18} Suppose that the $11****$ and $13****$ slices are good with types $xyzw$ and $x'y'z'w'$ respectively. If $x=x'$, then the $1*x***$ slice has defect at least $30$, and the $1*X***$ slice has defect at least $8$. Also, the $1**1**$, $1**3**$, $1***1*$, $1***3*$, $1****1$, $1****3$ slices have defect at least $6$. In particular, the total defect of slices beginning with $1*$ is at least $74$.
\end{lemma}

\begin{proof} Observe from the $11****, 13****$ hypotheses that $a(1*x***)=9$ and $a(1*X***)=4$, which gives the first two claims by Lemma \ref{defects}.  For the other claims, one sees from Lemma \ref{pqs} that the other six slices cannot be good; also, they have an $a$-value of $6$ and a $d$-value of at most $7$, and the claims then follow from Lemma \ref{defects}.
\end{proof}

Now we look at two diagonally opposite parallel good slices, such as $11****$ and $33****$. 

\begin{lemma}\label{l14} The $11****$ and $33****$ slices cannot both be good and of the same type.
\end{lemma}

\begin{proof} Suppose not.  By symmetry we may assume that $11****$ and $33****$ are of type $1111$. This excludes a lot of points from $22****$. Indeed, by connecting lines between the $11****$ and $33****$ slices, we see that the only points that can still survive in $22****$ are $221133, 221333, 221132, 223332$, and permutations of the last four indices. Double counting the lines $22133*$ and permutations we see that there are at most $12$ points one can place in the permutations of $221133, 221333, 221132$, and so the $22****$ slice has at most $16$ points. Meanwhile, the two five-dimensional slices $1*****, 3*****$ have at most $c'_{5,3} = 124$ points, and the other two four-dimensional slices $21****, 23****$ have at most $c'_{4,3} = 43$ points, leading to at most $16 + 124 * 2 + 43 * 2 = 350$ points in all, a contradiction.
\end{proof}

\begin{lemma}\label{l19} It is not possible for all four slices in a family to be good.
\end{lemma}

\begin{proof} Suppose not.  By symmetry we may assume that $11****, 13****, 31****, 33****$ are good.   By Lemma \ref{l14}, the $11****$ and $33****$ slices cannot be of the same type, and so they cannot both be of the opposite type to either $13****$ or $31****$. If $13****$ is not of the opposite type to $11****$, then by (a permutation of) Lemma \ref{l18}, the total defect of slices beginning with $1*$ is at least $74$; otherwise, if $13****$ is not of the opposite type to $33****$, then by (a permutation and reflection of) Lemma \ref{l18}, the total defect of slices beginning with $*3$ is at least $74$. Similarly, the total defect of slices beginning with $3*$ or $*1$ is at least $74$, leading to a total defect of at least $148$. But the total defect of all the corner slices is $2 \times 60 = 120$, a contradiction.
\end{proof}

\begin{corollary}\label{l20} At most one family can have a total defect of at least $38$.
\end{corollary}

\begin{proof} Suppose there are two families with defect at least $38$. The remaining thirteen families have defect at least $4$ by Lemma \ref{l19} and Lemma \ref{defects}, leading to a total defect of at least $38*2+13*4=128$. But the total defect is $2 \times 60 = 120$, a contradiction.
\end{proof}

Actually we can refine this:

\begin{proposition}  No family can have a total defect of at least $38$.
\end{proposition}

\begin{proof} Suppose for contradiction that the $ab****$ family (say) had a total defect of at least $38$, then by Corollary \ref{l20} no other families have total defect at least $38$.

We claim that the $**ab**$ family can have at most two good slices. Indeed, suppose the $**ab**$ has three good slices, say $**11**, **13**, **33**$. By Lemma \ref{l14}, the $**11**$ and $**33**$ slices cannot be of the same type, and so cannot both be of opposite type to $**13**$. Suppose $**11**$ and $**13**$ are not of opposite type. Then by (a permutation of) Lemma \ref{l18}, one of the families $a*b***, *ab***, **b*a*, **b**a$ has a net defect of at least $38$, contradicting the normalisation.

Thus each of the six families $**ab**, **a*b*, **a**b, ***ab*, ***a*b, ****ab$ have at least two bad slices. Meanwhile, the eight families $a*b***, a**b**, a***b*, a****b, *ab***, *a*b**, *a**b*, *a***b$ have at least one bad slice by Corollary \ref{l19}, leading to twenty bad slices in addition to the defect of at least $38$ arising from the $ab****$ slice. To add up to a total defect of $120$, we conclude from Lemma \ref{defects} that all bad slices outside of the $ab****$ family have a defect of four, with at most one exception; but then by Lemma \ref{l18} this shows that (for instance) the $1*1***$ and $1*3***$ slices cannot be good unless they are of opposite type. The previous argument then shows that the a*b*** slice cannot have three good slices, which increases the number of bad slices outside of $ab****$ to at least twenty-one, and now there is no way to add up to $120$, a contradiction. 
\end{proof}

\begin{corollary} Every family can have at most two good slices.
\end{corollary}

\begin{proof} If for instance $11****, 13****, 33****$ are all good, then by Lemma \ref{l14} at least one of $11****, 33****$ is not of the opposite type to $13****$, which by Lemma \ref{l18} implies that there is a family with a total defect of at least $38$, contradicting the previous proposition.
\end{proof}

From this corollary and Lemma \ref{defects}, we see that every family has a defect of at least $8$.  Since there are $15$ families, and $8 \times 15$ is exactly equal to $120$, we conclude

\begin{corollary}\label{coda} Every family has \emph{exactly} two good slices, and the remaining two slices have defect $4$.  In particular, by Lemma \ref{defects}, the bad slices must have statistics $(5,12,18,4,0)$, $(5,12,12,4,1)$, or $(6,8,12,8,0)$.  
\end{corollary}

We now limit how these slices can interact with good slices.

\begin{lemma}\label{goodgood}  Suppose that $1*1***$ is a good slice.
\begin{itemize}
\item[(i)] The $11****$ slice cannot have statistics $(6,8,12,8,0)$.
\item[(ii)] The $11****$ slice cannot have statistics $(5,12,12,4,1)$.
\item[(iii)] If the $11****$ slice has statistics $(5,12,18,4,0)$, then the $112***$ slice has statistics $(3,9,3,0)$.
\end{itemize}
\end{lemma}

\begin{proof}  This can be verified through computer search; there are $16$ possible configurations for the good slices, and one can calculate that there are $27520$ configurations for the $(5,12,12,4,1)$ slices, $4368$ configurations for the $(5,12,18,4,0)$ slices, and $80000$ configurations for the $(6,8,12,8,0)$ slices.  It is then a routine matter to inspect by computer all the potential counterexamples to the above lemma.
\end{proof}

\begin{corollary}\label{slic}  The $111***$ slice has statistics $(4,3,3,1)$, $(2,6,6,0)$, $(3,3,3,1)$, or $(1,6,6,0)$.
\end{corollary}

\begin{proof} From Corollary \ref{coda}, we know that at least one of the slices $13****, 31****, 11****$ are good.  If $11****$ or $1*1***$ is good, then the slice $111***$ has statistics $(4,3,3,1)$ or $(2,6,6,0)$, by Lemma \ref{sixt}.  By symmetry we may thus reduce to the case where $13****$ is good and $1*1***$ is bad.  Then by Lemma \ref{goodgood}, the $1*1***$ slice has statistics $(5,12,18,4,0)$ and the $121***$ slice has statistics $(3,9,3,0)$.  Since the $131***$ slice, as a side slice of the good $13****$ slice, has statistics $(4,3,3,1)$ or $(2,6,6,0)$, we conclude that the $111***$ slice has statistics $(1,6,6,0)$ or $(3,3,3,1)$, and the claim follows.
\end{proof}

\begin{corollary}\label{slic2} All corner slices have statistics $(6,12,18,4,0)$ or $(5,12,18,4,0)$.
\end{corollary}

\begin{proof}  Suppose first that a corner slice, say $11****$ has statistic $(6,8,12,8,0)$.  Then $111***$ and $113***$ contain one ``$d$'' point each, and have six ``$a$'' points between them, so by Corollary \ref{slic}, they both have statistic $(3,3,3,1)$.  This forces the $1*1***$, $1*3***$ slices to be bad, which by Corollary \ref{coda} forces the $3*1***,3*3***$ slices to be good.  This forces the $311***, 313***$ slices to have statistics either $(2,6,6,0)$ or $(4,3,3,1)$.  But the $311***$ slice (say) cannot have statistic $(4,3,3,1)$, since when combined with the $(3,3,3,1)$ statistics of $111***$ would give $a(*11***)=7$, which contradicts Corollary \ref{coda}; thus the $311***$ slice has statistic $(2,6,6,0)$, and similarly for $331***$.  But then $a(3*1***)=4$, which again contradicts Corollary \ref{coda}.

Thus no corner slice has statistic $(6,8,12,8,0)$.  Now suppose that a corner slice, say $11****$ has statistic $(5,12,12,4,1)$.  By Lemma \ref{goodgood}, the $1*1***, 1*3***$ slices are bad, so by repeating the preceding arguments we conclude that the $311***, 313***$ slices have statistics $(2,6,6,0)$ or $(4,3,3,1)$; in particular, their $a$-value is even.  However, the $*11***$ and $*13***$ slices are bad by Lemma \ref{goodgood}, and thus have an $a$-value of $5$; thus the $111***$ and $113***$ slices have an odd $a$-value.  Thus forces $a(11****)$ to be even; but it is equal to $5$, a contradiction.
\end{proof}

From this and Lemma \ref{dci}, we see that $A$ has statistics $(22,72,180,80,0,0,0)$.  In particular, we have $2\alpha_2(A)+\alpha_3(A) = 2$, which by double counting (cf. \eqref{alpha-1}) shows that for every line of the form $11122*$ (or a reflection or permutation thereof) intersects $A$ in exactly two points.  Note that such lines connect a ``$d$'' point to two ``$c$'' points.

Also, we observe that two adjacent ``$d$'' points, such as $111222$ and $113222$, cannot both lie in $A$; for this would force the $*13***$ and $*11***$ slices to have statistics $(4,3,3,1)$ or $(3,3,3,1)$ by Corollary \ref{slic}, which forces $a(*1****)=6$, and thus $*1****$ must be good by Corollary \ref{slic2}; but this contradicts Lemma \ref{sixt}.  Since $\alpha_3(A)=1/2$, we conclude that given any two adjacent ``$d$'' points, exactly one of them lies in $A$.  In particular, the d points of the form $***222$ consist either of those strings with an even number of $1$s, or those with an odd number of $1$s.

Let's say it's the former, thus the set contains $111222, 133222$, and permutations of the first three coordinates, but omits $113222, 333222$ and permutations of the first three coordinates. Since the ``$d$'' points $113222, 333222$ are omitted, we conclude that the ``$c$'' points $113122, 113322, 333122, 333322$ must lie in the set, and similarly for permutations of the first three and last three coordinates. But this gives at least $15$ of the $16$ ``$c$'' points ending in $22$; by symmetry this leads to $225$ $c$-points in all; but $c(A)=180$, contradiction.  This (finally!) completes the proof that $c'_{6,3}=353$.





\begin{thebibliography}{10}

\bibitem{ajtai}  M. Ajtai, E. Szemer\'edi, \emph{Sets of lattice points that form no squares}, Studia Scientiarum Mathematicarum Hungarica, \textbf{9} (1974-1975), 9--11. 

\bibitem{austin}  T. Austin, \emph{Deducing the density Hales-Jewett theorem from an infinitary removal lemma}, preprint, available at {\tt arxiv.org/abs/0903.1633}.

\bibitem{beck} J. Beck, Combinatorial Games: Tic-Tac-Toe Theory. Cambridge University Press, 2008, Cambridge.

\bibitem{behrend}
F. Behrend, \emph{On the sets of integers which contain no three in arithmetic progression}, Proceedings of the National Academy of Sciences \textbf{23} (1946), 331–-332.

\bibitem{Brower}
A. Brower, {\tt www.win.tue.nl/$\sim$aeb/codes/binary-1.html}.

\bibitem{chandra}
A. Chandra, \emph{On the solution of Moser's problem in four dimensions}, Canad. Math. Bull. \textbf{16} (1973), 507--511.

\bibitem{chvatal1} V. Chv\'{a}tal, \emph{Remarks on a problem of Moser}, Canad. Math. Bull., \textbf{15} (1972) 19--21.

\bibitem{chvatal2} V. Chv\'{a}tal, \emph{Edmonds polytopes and a hierarchy of combinatorial problems}, Discrete Math. \textbf{4} (1973) 305--337.

\bibitem{elkin}
M. Elkin, \emph{An Improved Construction of Progression-Free Sets}, preprint.

\bibitem{fuji}
K. Fujimura, {\tt www.puzzles.com/PuzzlePlayground/CoinsAndTriangles/CoinsAndTriangles.htm}

\bibitem{fk1} H. Furstenberg, Y. Katznelson, \emph{A density version of the Hales-Jewett theorem for $k = 3$}, Graph Theory and Combinatorics (Cambridge, 1988). Discrete Math. \textbf{75} (1989), 227–-241.

\bibitem{fk2} H. Furstenberg, Y. Katznelson, \emph{A density version of the Hales-Jewett theorem}, J. Anal. Math. \textbf{57} (1991), 64–-119. 

\bibitem{kra}
D. Geller, I. Kra, S. Popescu, S. Simanca, \emph{On circulant matrices}, {\tt www.math.sunysb.edu/$\sim$sorin/eprints/circulant.pdf}

\bibitem{greenwolf}
B. Green, J. Wolf, \emph{A note on Elkin's improvement of Behrend's construction}, preprint, available at {\tt arxiv.org/abs/0810.0732}.

\bibitem{heule} M. Heule, presentation at {\tt www.st.ewi.tudelft.nl/sat/slides/waerden.pdf}

\bibitem{komlos}
J. Koml\'{o}s, solution to problem P.170 by Leo Moser, Canad. Math. Bull. \textbf{15} (1972), 312--313, 1970.


\bibitem{markstrom} K. Markstr\"om, {{\tt abel.math.umu.se/$\sim$klasm/Data/HJ/}}

\bibitem{moser} L. Moser, Problem P.170 in Canad. Math. Bull. \textbf{13} (1970), 268.   

\bibitem{mcc} R. McCutcheon, \emph{The conclusion of the proof of the density Hales-Jewett theorem for $k=3$}, unpublished. 

\bibitem{obryant}
K. O'Bryant, \emph{Sets of integers that do not contain long arithmetic progressions}, preprint, available at {\tt arxiv.org/abs/0811.3057}.

\bibitem{oeis}
N. J. A. Sloane, Ed. (2008), The On-Line Encyclopedia of Integer Sequences, {\tt www.research.att.com/$\sim$njas/sequences/}

\bibitem{potenchin}
A. Potechin, \emph{Maximal caps in $AG(6, 3)$}, Des. Codes Cryptogr., \textbf{46} (2008), 243--259.

\bibitem{poly} D.H.J. Polymath, \emph{A new proof of the density Hales-Jewett theorem}, preprint, available at {\tt arxiv.org/abs/0910.3926}.

\bibitem{polywiki} D.H.J. Polymath, {\tt michaelnielsen.org/polymath1/index.php?title=Polymath1}

\bibitem{rankin} 
R. A. Rankin, \emph{Sets of integers containing not more than a given number of terms in arithmetical progression}, Proc. Roy. Soc. Edinburgh Sect. A \textbf{65} (1960/1961), 332–-344. 

\bibitem{roth}
K. Roth, \emph{On certain sets of integers, I}, J. Lond. Math. Soc. \textbf{28} (1953), 104-–109.


\bibitem{shelah} S. Shelah, \emph{Primitive recursive bounds for van der Warden numbers}, J. Amer. Math. Soc. \textbf{28} (1988), 683-–697.

\bibitem{sperner} 
E. Sperner, \emph{Ein Satz \"uber Untermengen einer endlichen Menge}, Mathematische Zeitschrift \textbf{27} (1928), 544-–548.

\bibitem{szem}
E. Szemer\'edi, \emph{On sets of integers containing no $k$ elements in arithmetic progression}, Acta Arithmetica \textbf{27} (1975), 199-–245.

\end{thebibliography}
\end{document}